\documentclass[a4paper,dvipsnames]{article}
\usepackage[utf8]{inputenc}
\usepackage[a4paper,top=3cm,bottom=2cm,left=3cm,right=3cm,marginparwidth=1.75cm]{geometry}

\usepackage{booktabs} 
\usepackage{setspace} 
\usepackage{cancel}   
\usepackage{enumitem} 
\usepackage{lipsum}   
\usepackage{etoolbox} 
\usepackage[margin=2cm]{caption} 
\usepackage{multirow} 
\usepackage{mathtools}

\usepackage{bbm}
\usepackage{amssymb,amsmath,amsthm}
\usepackage[notref, notcite, final]{showkeys}

\usepackage{authblk} 

\allowdisplaybreaks 

\usepackage[colorlinks, bookmarksnumbered, bookmarks]{hyperref} 

\usepackage{xcolor}
\usepackage{txfonts} 

\usepackage{comment}

\usepackage[backend=biber, hyperref=true, style=numeric, sorting=nyt, doi=true, isbn=false, url=false, maxbibnames=9, sortcites=true]{biblatex}
\renewenvironment{abstract}
  {\quotation
  {\bfseries\noindent{\abstractname:}}}
  {\endquotation}

\theoremstyle{definition} 
\newtheorem{theorem}{Theorem}[section] 
\newtheorem{definition}[theorem]{Definition}

\newtheorem{proposition}[theorem]{Proposition}

\newtheorem{assumption}[theorem]{Assumption}
\newtheorem{rmk_temp}[theorem]{Remark}
\numberwithin{equation}{section} 

\newenvironment{remark}
  {\pushQED{\qed}\begin{rmk_temp}}
  {\popQED\end{rmk_temp}}

\newbool{draftversion}
\NewDocumentEnvironment{remarkdraft}{+b}
  {\ifbool{draftversion}{\begin{remark}{\color{red}#1}\end{remark}}{}}{}
\boolfalse{draftversion}   

\newcommand{\N}{\mathbb{N}}
\newcommand{\Z}{\mathbb{Z}}
\newcommand{\R}{\mathbb{R}}
\newcommand{\C}{\mathbb{C}}
\newcommand{\A}{\mathbb{A}}

\newcommand{\eps}{{\varepsilon}}

\DeclareMathOperator{\Div}{div}

\let\oldsqrt\sqrt
\def\sqrt{\mathpalette\DHLhksqrt}
\def\DHLhksqrt#1#2{%
\setbox0=\hbox{$#1\oldsqrt{#2\,}$}\dimen0=\ht0
\advance\dimen0-0.2\ht0
\setbox2=\hbox{\vrule height\ht0 depth -\dimen0}%
{\box0\lower0.4pt\box2}}

\AtBeginDocument{

}

\title{\LARGE\MakeUppercase{\textbf{Operator aspects of wave propagation through periodic media
}}}

\author[1]{Kirill Cherednichenko} 
\author[2]{Yi-Sheng Lim}

\affil[1]{Department of Mathematical Sciences, University of Bath, Claverton Down, Bath BA2 7AY, \newline United Kingdom (Email: kc525@bath.ac.uk)}

\affil[2]{Department of Mathematics, Texas A\&M University, College Station, TX 77843, \newline United States (Email: yishenglimysl@tamu.edu)}

\date{}

\begin{document}

\maketitle

\vspace{-0.8cm}

\begin{abstract}
    Recent results in quantitative homogenisation of the wave equation with rapidly oscillating coefficients are discussed from the operator-theoretic perspective, which views the solution as the result of applying the operator of hyperbolic dynamics, i.e. the unitary group of a self-adjoint operator on a suitable Hilbert space. A prototype one-dimensional example of utilising the framework of Ryzhov boundary triples is analysed, where operator-norm resolvent estimates for the problem of classical moderate-contrast homogenisation are obtained. By an appropriate ``dilation" procedure, these are shown to upgrade to second-order (and more generally, higher-order) estimates for the resolvent and the unitary group describing the evolution for the related wave equation.   
    

    \vskip 0.5cm
    

    \vskip 0.5cm

\end{abstract}

\onehalfspacing
\section{Introduction}

\label{intro_sec}

\subsection{Homogenisation as a tool to study long waves}

This article is a survey of recent advances in homogenisation of the wave equation (WE), that is, in the study of approximating the effective \textit{transport} properties of a highly heterogeneous medium by those of a homogeneous one. In its basic form, homogenisation is introduced by considering the following initial-value problem for the WE:
%
%
\begin{align}\label{eqn:intro_wave_eqn}
    \begin{cases}
        \partial_{tt} u_\eps - \Div \left( \A_\eps \nabla u_\eps  \right) = f &\text{in $\R\times \R^d$,} \\
        u_\eps(0,\cdot) = u_{\text{init}}, \quad 
        \partial_t u_\eps (0,\cdot) = v_{\text{init}}
        &\text{on $\{ t=0 \} \times \R^d$,}
    \end{cases}
\end{align}
where $f$, $\A_\eps$, $u_{\text{init}}$, and $v_{\text{init}}$ are given, and a solution $u_\eps(t,x)$ is sought in an appropriate sense. The coefficient matrix $\A_\eps$ is the $\eps$-rescaling of a prescribed $\Z^d$-periodic symmetric positive-definite matrix-valued function $\mathbb{A}:\R^d \rightarrow \R_\text{sym}^{d\times d},$ namely $\A_\eps(x) = \A(x/\eps)$. The parameter $\eps>0$ thus encodes the heterogeneity of the medium. The starting point of our discussion is the following result.

\begin{theorem}\label{thm:basic_wave_homog}
    \cite[Theorem 12.6]{cioranescu_donato}.
    Let $\Omega \subset \R^d$ be a bounded Lipschitz domain. 
    Suppose that $\A$ is positive-definite uniformly in $x \in \R^d$, and has $L^\infty$ entries. Suppose that $f \in L^2((0,T) \times \Omega)$, $u_{\text{init}} \in H_0^1(\Omega)$, $v_{\text{init}} \in L^2(\Omega)$. Let $u_\eps$ be the solution of \eqref{eqn:intro_wave_eqn} on $\R\times\Omega$. Then, for each $T>0$, one has 
    \begin{alignat}{3}
        &u_\eps \rightharpoonup u_{\hom} 
        &&\text{weakly$^*$ in $L^\infty\bigl((0,T); H_0^1(\Omega)\bigr)$}, \label{basic_conv_uhom}\\
        &u_\eps' \rightharpoonup u_{\hom} 
        &&\text{weakly$^*$ in $L^\infty\bigl((0,T); L^2(\Omega)\bigr)$}, \\
        &\A_\eps \nabla u_\eps \rightharpoonup \A^{\hom} \nabla u_{\hom} \qquad 
        &&\text{weakly in $\bigl(L^2\bigl((0,T)\times\Omega\bigr)\bigr)^d$,}
    \end{alignat}
    where $u_{\hom}$ is the solution to the homogenised problem
    \begin{align}\label{eqn:homogenised_wave_eqn}
        \begin{cases}
            \partial_{tt} u_{\hom} - \Div(\A^{\hom} \nabla u_{\hom} ) = f, &\text{in $(0,T)\times\Omega$,}\\
            u_{\hom}(0,\cdot) = u_{\text{init}}, \quad 
            \partial_t u_{\hom}(0,\cdot) = v_{\text{init}}
            &\text{on $\{t=0 \}\times\Omega$.}
        \end{cases}
    \end{align}
\end{theorem}


The matrix $\A^{\hom}$ appearing in Theorem \ref{thm:basic_wave_homog} is called the \textit{homogenised tensor} -- it is constant in space, representing an effective homogeneous medium. Theorem \ref{thm:basic_wave_homog} is by now a classical result, and can be proven by various means, for instance, by two-scale convergence (in the elliptic setting) and Galerkin approximation (in the hyperbolic setting) \cite[Chapter 12]{cioranescu_donato}, by a two-scale expansion \cite[Chapter 4]{bakhvalov_panasenko}, or by G-convergence \cite[Chapter 5]{zhikov}. However, Theorem \ref{thm:basic_wave_homog} is insufficient from wave-propagation perspective because (a) it is a qualitative result (i.e.~no rate of convergence), and (b) it is only a finite-time approximation. While there has been substantial efforts over the past two decades to address (a), much of the literature has been focused on the stationary setting (see for instance \cite[Sect~1.3]{simplified_method} for a recent overview). As far as time-dependent equations are concerned, much of the activity lies in the parabolic setting, mainly due to the fact that the fundamental solution exhibits nice decay properties (see e.g.~\cite[Chapter~2]{zhikov} or \cite[Chapter~8--9]{qshomo_book}). The goal of this review is to draw awareness to the literature in the \textit{hyperbolic} case, and specifically the WE. We shall avoid any discussion on boundary effects, and focus on the full-space setting $\Omega = \R^d$.

%
%

\subsection{Operator perspective}
A number of competing approaches have been developed for the analysis of the behaviour of initial value problems for the wave equation with rapidly oscillating periodic coefficients. Naturally, the results obtained differ in terms of the balance between the approximation error, the time interval (expressed in terms of the parameter $\varepsilon$), and the quality of the data (the initial conditions and the right-hand side of the equation) -- the latter can usually be expressed in terms of the behaviour of the spatial and temporal Fourier transforms of the data for large values of the corresponding Fourier parameters. Usually, homogenisation estimates are sought with respect to the $L^2$ norm and the $H^1$ (``energy") norm. 

The small parameter $\varepsilon$ (appearing as the period of the coefficients in the equation) represents the ratio between the physical length-scale of material property oscillations and another length -- the latter is then much larger than the former, so we shall refer to it as ``macroscopic". The natural choice of the macroscopic length depends on the problem in question, and so the asymptotic analysis as $\varepsilon\to0$ corresponds to selecting a class of solutions of the original equation that are ``close" (in some sense, which is to be specified as part of the proof of error estimates) to some solutions of an equation with non-oscillatory coefficients, which we will refer to as ``homogenised". In the study of wave propagation through a medium occupying a large part of space, when mathematically the whole-space set-up appears to be a plausible model of the physical process and assuming that the material properties of the medium are time-independent, the standard spectral analysis based on the temporal Fourier transform is a natural step towards introducing the macroscopic length. In effect, the special solutions mentioned above are in this case finite-energy combinations of monochromatic waves with frequencies that render the corresponding wavelengths controllably large compared to the period of the material oscillations (which plays the role of a ``microscale"). When considering the Cauchy problem, the time interval over which the corresponding solutions to the original heterogeneous and the homogenised wave equations are controllably close to each other (in terms of some order of smallness with respect to $\varepsilon$) will depend on the degree of dispersion of the wave energy into modes not captured by the homogenised equation. 

The above discussion only involves two length-scales and hence one small parameter, which we have labelled by $\varepsilon.$ It is then implicit that the material properties do not vary much across the period -- in quantitative terms, the product of the period and the spatial gradient of the material coefficients is uniformly small relative to $\varepsilon^{-\gamma}$ for any $\gamma>0.$ For example, for a two component medium, one has two natural lengths, namely the wavelengths (at a given frequency $\omega$) $\lambda_-=2\pi c_1/\omega,$ $\lambda_2(\omega)=2\pi c_2/\omega$ the two components -- here, $c_1,$ $c_2$ are the corresponding wavespeeds. When $\lambda_-/\lambda_2$ is close to unity (the case of ``moderate contrast"), the parameter $\varepsilon$ (the ratio of the period and, say $\lambda_-$) is not too different from $\delta:=\varepsilon(\lambda_-/\lambda_2)$ and so the passage to the limit as $\varepsilon\to0,$ $\delta\to0$ is accomplished without specifying a ``path" in the $(\varepsilon, \delta)$ parameter space -- this is the scenario that most existing literature focusses on and that we also refer to as ``classical". Increasing the ratio $\lambda_-/\lambda_2$ or its inverse leads to deterioration of the ``classical" error estimate, due to the waves with shorter lengths being admitted by one of the components, which results in a non-classical, two-scale, wave picture on the macroscale, i.e., on the scale of the larger wavelength of the two. In practical terms, in any heterogeneous medium there is a certain amount of ``non-classical" behaviour due to the length separation between the wave lengths involved, e.g., $\lambda_-$ and $\lambda_2$ in the above case of a two-component medium. In this review we focus on the case of moderate contrast.     

Consider a Hilbert space $H$ and an unbounded positive-definite self-adjoint operator $A$ on $H.$  The Cauchy problem for the WE associated with $A$ consists in finding 
$u\in C^2(\mathbb R; H)$ such that
\begin{equation}
\partial_{tt}u-Au=f(t) \quad \forall t\in{\mathbb R},\qquad u(0)=u_\text{init},\quad u_t(0)=v_\text{init}
\label{Cauchy}
\end{equation}
for given $f\in C(\mathbb R; H),$ $u_\text{init} \in{\rm dom}(A),$ $v_\text{init}\in H,$ see e.g. \cite[Section 6.2]{konrad_book}. If 
$f(t)\in{\rm dom}(A^{1/2})$ for all $t\in{\mathbb R},$ $u_\text{init}\in{\rm dom}(A),$ $v_\text{init}\in {\rm dom}(A^{1/2}),$ the solution to (\ref{Cauchy}) is given by 
\begin{equation}
u(t)=\cos\bigl(A^{1/2}t\bigr)u_\text{init}+A^{-1/2}\sin\bigl(A^{1/2}t\bigr)v_\text{init}+\int_0^t A^{-1/2}\sin\bigl(A^{1/2}(t-s)\bigr)f(s)ds.
\label{hyperbolic_formula}
\end{equation}

Many problems of wave propagation (in the physical contexts of acoustics, linear (i.e., small-displacement) elastodynamics,  electromagnetics) can be written in the form (\ref{Cauchy}), where $A$ is a second-order linear differential operator.

As a prototype example, in Sections \ref{prototype}--\ref{second_order_dynamics} we shall consider the operator $A$ on the Hilbert space $H=L^2(0, l)\oplus L^2(l,1)$ (for fixed $l\in(0,1)$) with domain 
\begin{equation*}
\begin{aligned}
{\rm dom}(A)=\bigl\{u=u_-\oplus u_+&\in H^2(0,l)\oplus H^2(l, 1):
u_-(0)=u_+(1), u_-(l)=u_+(l), u_-'(0)=u_+'(1), u_-'(l)=u_+'(l)\bigr\}
\end{aligned}
\end{equation*}
defined by the differential expression $(au')',$ where the coefficient $a$ is piecewise-constant, $a(x)=a_-,$ $x\in(0,l),$  $a(x)=a_2,$ $x\in(l,1).$

\section{Improving the basic homogenisation result}\label{sect:improving_basic_homo}

To unify our discussion, let us view \eqref{eqn:intro_wave_eqn} and \eqref{eqn:homogenised_wave_eqn} from an operator perspective: Consider
\begin{align}
    \mathcal{A}_\eps = -\Div\left( \A_\eps \nabla \right)
    \qquad \text{and} \qquad
    \mathcal{A}^{\hom} = -\Div\left( \A^{\hom} \nabla \right)
\end{align}
as unbounded self-adjoint operators on $H = L^2(\R^d)$, then the Cauchy problem for the heterogeneous WE \eqref{eqn:intro_wave_eqn} is solved by (cf. (\ref{hyperbolic_formula}))
\begin{align}\label{eqn:u_eps_abstract}
    u_\eps(t) = \cos\left( \mathcal{A}_\eps^{1/2} t \right) u_\text{init}
    + \mathcal{A}_\eps^{-1/2} \sin\left( \mathcal{A}_\eps^{1/2} t \right) v_\text{init}
    + \int_0^t \mathcal{A}_\eps^{-1/2} \sin\left( \mathcal{A}_\eps^{1/2} (t-s) \right) f(s)ds,
\end{align}
and the Cauchy problem for the homogeneous WE \eqref{eqn:homogenised_wave_eqn} is solved by
\begin{equation}\label{eqn:u_hom_abstract}
\begin{split}
    u_{\hom}(t) &= \cos\left( (\mathcal{A}^{\hom})^{1/2} t \right) u_\text{init}
    + (\mathcal{A}^{\hom})^{-1/2} \sin\left( (\mathcal{A}^{\hom})^{1/2} t \right) v_\text{init} \\
    &\qquad\qquad + \int_0^t (\mathcal{A}^{\hom})^{-1/2} \sin\left( (\mathcal{A}^{\hom})^{1/2} (t-s) \right) f(s)ds.
\end{split}
\end{equation}

\subsection{The spectral germ approach}
In the seminal paper \cite{birman_suslina_2004}, Birman and Suslina introduced a novel approach to the study of highly-oscillatory media which we will henceforth call the ``spectral germ approach". It is based on the observation that homogenisation is essentially a task of approximating the periodic operator $\mathcal{A}_{\eps=1}$ near the bottom of the spectrum (``threshold") $z=0$. This is done through a Floquet-Bloch analysis of $\mathcal{A}_\eps$, which we shall briefly describe below.\footnote{As a historical note, the use of Floquet-Bloch analysis in the context of homogenisation can be traced back to Conca and \newline Vanninathan \cite{conca_vanninathan1997}, but the authors did not pursue the goal of improving the convergences in Theorem \ref{thm:basic_wave_homog}}

Applying the Floquet-Bloch-Gelfand transform to $\mathcal{A}_\eps$, one obtains the operator family $\{\eps^{-2} \mathcal{A}_\chi \}_{\chi \in [-\pi,\pi]^d}$. The operators $\mathcal{A}_\chi$ act on $L^2([0,1]^d)$, and are given by the differential expression $(\nabla_y + i\chi)^\ast \mathbb{A}(y)(\nabla_y + i\chi)$. As a result, the period $\eps$ now appears as a scaling factor $\eps^{-2}$, the dependence on the ``quasi-momentum/wavevector" parameter $\chi$ is analytic \cite{kato_book, reed_simon4}, and the spectrum $\sigma(\mathcal{A}_\chi)$ is discrete. As noted above, we are interested in the bottom of the spectrum, and this corresponds to the first eigenvalue of $\mathcal{A}_{\chi=0} = -\Div_y ( \mathbb{A}(y) \nabla_y)$, thus one seeks to approximate the first eigenspace of $\mathcal{A}_\chi$ near $\chi=0$. The key object in this approximation is the so-called ``spectral germ", which in this case is simply the number $(i\chi)^\ast \mathbb{A}^{\hom}(i\chi)$ (and in the abstract theory, viewed as an auxiliary operator on $\ker(\mathcal{A}_{\chi=0}) = \C$). By a careful reconstruction of $\mathcal{A}^{\hom}$ from the germ $(i\chi)^\ast \mathbb{A}^{\hom}(i\chi)$, one obtains the following \textit{norm-resolvent} estimate
\begin{align}\label{eqn:norm_resolvent_estimate}
    \bigl\| (\mathcal{A}_\eps + I)^{-1} - (\mathcal{A}^{\hom} + I)^{-1}\bigr\|_{L^2(\R^d) \rightarrow L^2(\R^d)} \leq C \eps, \quad \text{where $C>0$ is independent of $\eps$.}
\end{align}
or equivalently,
\begin{align}
    \| u_\eps - u_{\hom} \|_{L^2(\R^d)} \leq C \eps \, \| f \|_{L^2(\R^d)}, \qquad \text{where $C>0$ is independent of $\eps$ and $f$.}
\end{align}
Operator-norm/uniform estimates such as \eqref{eqn:norm_resolvent_estimate} first appeared in \cite{birman_suslina_2004}. It turns out that \eqref{eqn:norm_resolvent_estimate} is order-sharp.

\begin{remark}\label{rmk:ker_A0}
    The space of constant functions (in $y$), $\C = \ker(\mathcal{A}_{\chi=0})$ play a key role in \textit{all} approaches to homogenisation, and appears in different guises. For instance, ensures that the two-scale expansion \cite{bakhvalov_panasenko} has a leading-order term that is independent of the microscopic variable.
\end{remark}

\begin{remark}
    For the remainder of this section, the constant $C>0$ will be independent of $\eps$ and $t$. 
\end{remark}

The spectral germ approach has since undergone several developments, and we shall now discuss its extension to the WE. Due to the operator representation \eqref{eqn:u_eps_abstract} and \eqref{eqn:u_hom_abstract}, we break our investigation into $\cos\bigl( \mathcal{A}_\eps^{1/2} t \bigr)$, $\mathcal{A}_\eps^{-1/2} \sin{\bigl(t\mathcal{A}_\eps^{1/2}\bigr)}$, and $\int_0^t \mathcal{A}_\eps^{-1/2} \sin\bigl( \mathcal{A}_\eps^{1/2} (t-s) \bigr) f(s)ds$. We begin with the operator $\cos\bigl( \mathcal{A}_\eps^{1/2} t \bigr)$. In \cite[Sect 13.1]{birman_suslina_2009_hyperbolic}, Birman and Suslina proved that for $0 \leq s \leq 2$, we have
\begin{align}\label{eqn:birman_suslina_cosine}
    \left\| \, \cos\left( \mathcal{A}_\eps^{1/2} t \right) - \cos\left( (\mathcal{A}^{\hom})^{1/2} t \right) \, \right\|_{H^s(\R^d) \rightarrow L^2(\R^d)} \leq C_s \eps^{s/2} (1 + |t|)^{s/2}, \qquad t \in \R.
\end{align}
Then, due to the operator identity $\mathcal{A}_\eps^{-1/2} \sin\bigl(t\mathcal{A}_\eps^{1/2}\bigr) = \int_0^t \cos{(t\mathcal{A}_\eps^{1/2})} dt$, one obtains
\begin{align}\label{eqn:birman_suslina_sine}
    \left\| \mathcal{A}_\eps^{-1/2} \sin{\left(t\mathcal{A}_\eps^{1/2}\right)} - (\mathcal{A}^{\hom})^{-1/2} \sin\left( (\mathcal{A}^{\hom})^{1/2} t \right) \right\|_{H^s(\R^d) \rightarrow L^2(\R^d)} \leq C \eps^{s/2} (1 + |t|)^{1+s/2}, \qquad t \in \R.
\end{align}

Note that if the initial datum $u_\text{init}$ or $v_\text{init}$ is only $L^2(\R^d)$ (the $s=0$ case), then the above estimates are useless. Indeed, norm-resolvent convergence \eqref{eqn:norm_resolvent_estimate} only guarantees $\| g(\mathcal{A}_\eps) - g(\mathcal{A}^{\hom}) \|_{op} \rightarrow 0$, for $g$ that is continuous on $\R$ and vanishes at infinity \cite[Theorem VIII.20]{reed_simon1}, which $g_t(\lambda) = \cos{(\lambda^{1/2} t)}$ does not satisfy.

Building on \cite{birman_suslina_2009_hyperbolic} was a series of works to confirm or improve upon the estimates \eqref{eqn:birman_suslina_cosine}-\eqref{eqn:birman_suslina_sine}. In \cite{dorodnyi_suslina_2018_hyperbolic}, Dorodnyi and Suslina verified that \eqref{eqn:birman_suslina_cosine} is sharp in the sense of the regularity of the initial data. That is, the $H^s\rightarrow L^2$ norm cannot be replaced by $H^r\rightarrow L^2$ with $r<s$, while maintaining the same RHS.\footnote{This implies that \eqref{eqn:birman_suslina_cosine} is \textbf{order}-sharp for $s<2$.} On the other hand, it turns out that \eqref{eqn:birman_suslina_sine} could be improved. Focusing on the $s=1$ case (in which \eqref{eqn:birman_suslina_sine} says that we have a valid approximation for times $t \leq \eps^{-1/3 + \delta}$), Meshkova \cite{meshkova_2021_hyperbolic} showed that
\begin{align}\label{eqn:birman_suslina_sine_sharp}
    \left\| \mathcal{A}_\eps^{-1/2} \sin{\left(t\mathcal{A}_\eps^{1/2}\right)} - (\mathcal{A}^{\hom})^{-1/2} \sin\left( (\mathcal{A}^{\hom})^{1/2} t \right) \right\|_{H^1(\R^d) \rightarrow L^2(\R^d)} \leq C \eps\bigl(1 + |t|\bigr), \qquad t \in \R,
\end{align}
in which the sharpness (in the same sense) is verified in \cite{dorodnyi_suslina_2018_hyperbolic}.

Overall, the spectral germ approach gives the following result.

\begin{theorem}\label{thm:spectral_germ_overall}
    \cite[Theorem 12.1]{dorodnyi_suslina_2018_hyperbolic} 
    Let $0 \leq s \leq 2$ and $0 \leq r \leq 1$.\footnote{While $A^{1/2}$ is unbounded with domain $\text{dom}(A^{1/2}) = H^1$, we can extend $A^{-1/2}\sin{(A^{1/2}t)}$ to a bounded operator on $L^2$, for each $t$. Thus, it makes sense to speak about $A^{-1/2}\sin{(A^{1/2}t)} v_\text{init}$, where $v_\text{init}$ lies in $H^r$, $0 \leq r \leq 1$. A similar remark applies to $u_\text{init}$ and $f$.} If $u_\text{init} \in H^s(\R^d)$, $v_\text{init} \in H^r(\R^d)$, and $f \in L_{\text{loc}}^1(\R;H^r(\R^d))$, then there exists positive constants $C_s$ and $C_r$, independent of $\eps$, such that for $t \in \R$,
    \begin{align}\label{eqn:spectral_germ_overall}
        \bigl\| u_\eps(t) - u_{\hom}(t) \bigr\|_{L^2(\R^d)} \leq 
        C_s \eps^{s/2} (1 + |t|)^{s/2} \| u_\text{init} \|_{H^s(\R^d)}
        + C_r \eps^r (1 + |t|) \left[ \| v_\text{init} \|_{H^r(\R^d)} + \| f \|_{L^1((0,t);H^r(\R^d))} \right].
    \end{align}
    Moreover, if we only have $u_\text{init}, v_\text{init} \in L^2(\R^d)$ and $f \in L_{\text{loc}}^1(\R;L^2(\R^d))$ (the case $s=r=0$), then
    \begin{align}\label{eqn:l2data_theorem}
        \bigl\| u_\eps(t) - u_{\hom}(t) \bigr\|_{L^2(\R^d)} \rightarrow 0, \qquad t \in \R.
    \end{align}
\end{theorem}

\begin{remark}
    The convergence \eqref{eqn:l2data_theorem} is a direct consequence of \eqref{eqn:norm_resolvent_estimate} and \cite[Theorem VIII.20(b)]{reed_simon1}.
\end{remark}

\paragraph*{Summary.} The spectral germ approach improves the basic homogenisation result (Theorem \ref{thm:basic_wave_homog}) by (a) upgrading the convergence to an operator-norm type, with an explicit rate, and (b) provides an estimate beyond a finite time window. Regarding (b), let us write \eqref{eqn:spectral_germ_overall} in the case of $u_\text{init} \in H^2({\mathbb R}^d)$, $v_\text{init} \in H^1({\mathbb R}^d)$, and $f \in L_{\text{loc}}^1\bigl(\R;H^1({\mathbb R}^d)\bigr)$, for reader's convenience:
\begin{align}\label{eqn:spectral_germ_overall2}
    \bigl\| u_\eps(t) - u_{\hom}(t)\bigr\|_{L^2({\mathbb R}^d)} \leq 
    C \eps (1 + |t|) \left[ \| u_\text{init} \|_{H^2({\mathbb R}^d)} + \| v_\text{init} \|_{H^1({\mathbb R}^d)} + \| f \|_{L^1((0,t);H^1({\mathbb R}^d))} \right].
\end{align}
That is, $u_{\hom}$ is a valid approximation of $u_\eps$ up to times $t\leq \eps^{-1+\delta}$.

One could wonder if the $\mathcal{O}(\eps(1+|t|))$ error in \eqref{eqn:spectral_germ_overall2} could be improved, if we are willing to restrict ourselves to smooth initial data $(u_\text{init}, v_\text{init}, f)$. We expect the answer to be \textit{no}, but to our knowledge there has been no proper discussion of this claim. If in addition to smooth data $(u_\text{init}, v_\text{init}, f)$, one adds more terms to $u_{\hom}(t)$ following the two-scale expansion \cite{bensoussan_lions_papanicolaou, bakhvalov_panasenko}, could the error be improved then? The answer is \textit{yes, to a certain extent} (times $t \leq \eps^{-2+\delta}$). This brings us to the findings of Allaire, Lamacz-Keymling, and Rauch \cite{allaire_lamacz_rauch_2022_crime_pays}, which we shall elaborate in the next section.

\subsection{Two-scale expansion and the secular growth problem}

To fix notation, we shall begin with a brief review of the classical two-scale expansion \cite{bensoussan_lions_papanicolaou, bakhvalov_panasenko} in the context of the WE \eqref{eqn:intro_wave_eqn}. We make the following assumptions on the initial data.

\begin{assumption}\label{assumption}
    $u_\text{init} = v_\text{init} = 0$, and $f \in C^\infty(\R;H^\infty(\R^d))$ with $\text{supp}(f)\subset [0,1]_t \times \R_x^d$.
\end{assumption}

\begin{definition}
    For two $k$-order tensors $A = (a_{i_1,\cdots,i_k})_{1\leq i_1,\cdots,i_k \leq d}$ and $B = (b_{i_1,\cdots,i_k})_{1\leq i_1,\cdots,i_k \leq d}$, their (full) tensor contraction is denoted $A \odot B = \sum_{i_1,\cdots,i_k} a_{i_1,\cdots,i_k}b_{i_1,\cdots,i_k}$. For matrices (i.e.~$k=2$), we write $A:B := A\odot B$. 
\end{definition}

\subsubsection*{Brief review of the (hyperbolic) two-scale expansion.} We seek an asymptotic expansion for $u_\eps$ in powers of $\eps$
\begin{align}\label{eqn:two_scale_expansion_wave1}
    u_\eps \sim u_0 + \eps u_1 + \eps^2 u_2 + \eps^3 u_3 + \cdots,
\end{align}
where we assume that each $u_n$ depend on $t$, and two spatial variables: a macroscopic (``slow") variable $x$, and a microscopic (``fast") variable $y$, which we will evaluate at $y = x/\eps$ (hence the term ``fast variable"). That is, 
\begin{align}\label{eqn:two_scale_expansion_wave2}
    u_\eps(t,x) \sim u_0\biggl(t,x,\frac{x}{\eps}\biggr) + \eps u_1\biggl(t,x,\frac{x}{\eps}\biggr) + \eps^2 u_2\biggl(t,x,\frac{x}{\eps}\biggr) + \eps^3 u_3\biggl(t,x,\frac{x}{\eps}\biggr) + \cdots.
\end{align}
We assume that each $u_j(t,x,y)$ is $\Z^d-$periodic in the $y-$variable.

For a function $u(t,x,y)$, write $\nabla_x$ and $\nabla_y$ for its derivatives in the variable $x$ and $y$ respectively. Then, $\Psi_\eps(x) := \Psi(x,x/\eps)$ gives $(\nabla \Psi_\eps)(x) = \eps^{-1}(\nabla_y \Psi)(x,x/\eps) + (\nabla_x \Psi)(x,x/\eps)$. Thus,
\begin{align}
    (\mathcal{A}_\eps \Psi_\eps)\biggl(x,\frac{x}{\eps}\biggr) = 
    \left[ \frac{1}{\eps^2} \mathcal{A}_{yy}\Psi_\eps 
    + \frac{1}{\eps} \mathcal{A}_{xy}\Psi_\eps 
    + \mathcal{A}_{xx}\Psi_\eps \right]\biggl(x,\frac{x}{\eps}\biggr),
\end{align}
where
\begin{align}
    \mathcal{A}_{yy} = - \Div_y \bigl( \A(y) \nabla_y  \bigr),
    \quad
    \mathcal{A}_{xy} = - \Div_x \bigl( \A(y) \nabla_y  \bigr) - \Div_y \bigl( \A(y) \nabla_x  \bigr),
    ~~\text{and} \quad
    \mathcal{A}_{xx} = - \Div_x \bigl( \A(y) \nabla_x  \bigr).
\end{align}
Here, $\nabla_y$ and $-\Div_y$ are equipped with periodic boundary conditions (as we are seeking $y-$periodic functions), and $\A(y)$ denotes the restriction of $\A:\R^d\rightarrow \R^{d\times d}_{\text{sym}}$ to the unit-cell $Y=[0,1]^d$. Applying the operator $\partial_{tt} + \eps^{-2} \mathcal{A}_{yy} + \eps^{-1} \mathcal{A}_{xy} + \mathcal{A}_{xx}$ to the RHS of \eqref{eqn:two_scale_expansion_wave1} (as a function in $(t,x,y)$), we obtain a formal expansion for the WE:
\begin{align}\label{eqn:wave_eqn_formal}
    \Biggl(\partial_{tt} + \frac{1}{\eps^2} \mathcal{A}_{yy} + \frac{1}{\eps} \mathcal{A}_{xy} + \mathcal{A}_{xx} \Biggr) \sum_{n=0}^{\infty} \eps^n u_n = f.
\end{align}

One then starts equating like powers of $\eps$, giving rise to a system of equations:
\begin{equation}\label{eqn:classical_alg_system}
    \begin{cases}
        \begin{alignedat}{4}
            &\mathcal{O}(\eps^{-2}) & &\mathcal{A}_{yy} u_0 &= 0, \\
            &\mathcal{O}(\eps^{-1}) & &\mathcal{A}_{yy} u_1 + \mathcal{A}_{xy} u_0 &= 0,\\
            &\mathcal{O}(1) &\partial_{tt} u_0 ~+~ &\mathcal{A}_{yy} u_2 + \mathcal{A}_{xy} u_1 + \mathcal{A}_{xx} u_0 \qquad&= f,\\
            &\mathcal{O}(\eps) &\partial_{tt} u_1 ~+~ &\mathcal{A}_{yy} u_3 + \mathcal{A}_{xy} u_2 + \mathcal{A}_{xx} u_1 &= 0,\\
            &\text{and so on...}\qquad
        \end{alignedat}
    \end{cases}
\end{equation}

The $\mathcal{O}(\eps^{-2})$ problem gives us $u_0 \in \ker(\mathcal{A}_{yy}) = \C_y$ (Remark \ref{rmk:ker_A0}). Thus, $u_0$ is constant in $y$, i.e.\,$u_0(t,x,y)=u_0(t,x)$. Following \cite{allaire_lamacz_rauch_2022_crime_pays}, we introduce a notation for projection of a function $y\mapsto u(t,x,y)$ onto $\C$:
\begin{definition}
    For a function $u(t,x,y)$, define the projection $\pi$ by $(\pi u)(t,x) = \int_Y u(t,x,y) dy$, and $\pi^\perp = I-\pi$.
\end{definition}

With this notation, we have
\begin{align}\label{eqn:u0_full}
    u_0(t,x,y) = \cancel{\pi^\perp u_0} + \pi u_0.
\end{align}

For the $\mathcal{O}(\eps^{-1})$ problem, since $\nabla_y u_0 = 0$, the term $\mathcal{A}_{xy}u_0 = \sum_{i,j} \partial_{y_i}\A_{ij}(y) \partial_{x_j} u_0 (x)$ exhibits a separation of variables, and thus we seek an ansatz of the form
\begin{align}\label{eqn:u1_full}
    u_1(t,x,y) 
    = \sum_{j=1}^d N_j(y) \frac{\partial u_0}{\partial x_j} (t,x) + \tilde{u}_1(t,x) 
    = \underbrace{\mathbf{N}(y) \cdot (\nabla_x u_0)(t,x)}_{= \pi^\perp u_1} + \underbrace{\tilde{u}_1(t,x)}_{= \pi u_1}.
\end{align}
This implies that $N_j(y)$ has to solve the cell-problem $\mathcal{A}_{yy} N_j = \nabla_y \cdot (\A(y)e_j)$, which has a unique solution in $\dot{H}^1_{\text{per}}(Y)$. $\mathbf{N} = \mathbf{N}^{(1)} = (N_1,\cdots,N_d)^\top$ is the (classical) first-order corrector. Here we define
\begin{definition}\label{defn:H1per_meanzero}
    $\dot{H}^1_{\text{per}}(Y)$ for the subspace of $H^1(Y)$ consisting of mean-zero periodic (in $y$) functions.
\end{definition}

For the $\mathcal{O}(1)$ problem, the Fredholm alternative asserts that for $u_2$ to be well-defined, we need to impose the condition $\int_Y (- \partial_{tt} u_0 - \mathcal{A}_{xy} u_1 - \mathcal{A}_{xx} u_0 + f)dy = 0$. This amounts to the following equation for $u_0(t,x)$ (``homogenised equation"):
\begin{align}\label{eqn:u0_hom}
    \partial_{tt}u_0 - \Div\bigl( \A^{\hom} \nabla u_0 \bigr) = f.
\end{align}

One proceeds down the system \eqref{eqn:classical_alg_system}, using information of $u_0,\cdots,u_{n-1}$ to determine $u_n$. For the general $\mathcal{O}(\eps^{n-2})$ problem, $n\geq 2$, we determine $\pi u_{n-2}$ through the well-posedness condition on $u_n$, and the $n$th-order corrector $\mathbf{N}^{(n)}$ through a separation of variables. This concludes the review.

\subsubsection*{The structure of the classical two-scale expansion.} In \cite[Sect~2,3]{allaire_lamacz_rauch_2022_crime_pays}, the authors propose an equivalent way of carrying out the classical two-scale expansion (see also \cite{bakhvalov_panasenko, kirill_valery_variational}). \textbf{Step 1.} Rather than going down the system \eqref{eqn:classical_alg_system} in increasing powers of $\eps$, we split the procedure into the terms $\pi^\perp u_n$ (``oscillatory hierarchy'') and the terms $\pi u_n$ (``non-oscillatory hierarchy"). \textbf{Step 2.} Write \eqref{eqn:wave_eqn_formal} as equation of formal series 
\begin{align}\label{eqn:wave_eqn_formal_v2}
    \Bigg(\partial_{tt} + \frac{1}{\eps^2} \mathcal{A}_{yy} + \frac{1}{\eps} \mathcal{A}_{xy} + \mathcal{A}_{xx} \Bigg) \sum_{n=0}^{\infty} \eps^n u_n = \sum_{n=-2}^\infty \eps^n w_n,
\end{align}
which is to be understood as a system of equations in like powers of $\eps$.
Observe that we now have to make a choice of distributing $f(t,x)$ into $w_n(t,x,y)$'s. Note that $w_n$ is further split into its oscillatory $\pi^\perp w_n$ and non-oscillatory $\pi w_n$ parts.

\textbf{Step 3.} We focus on the oscillatory hierarchy $(\pi^\perp u_n)$. \textit{Let us make the choice}
\begin{align}\label{eqn:choice_osc_classical}
    \pi^\perp w_n = 0, \qquad \text{for all} \qquad n\geq -2.
\end{align}
(which is natural because $f(t,x)$ does not depend on $y$). The remarkable fact is that this choice is equivalent (!) to the formal double series expansion which would get in the classical procedure \cite{bakhvalov_panasenko}
\begin{equation}\label{eqn:double_series}
    \begin{split}
        &\sum_{n=0}^\infty \eps^n u_n \sim \left( \sum_{j=0}^\infty \eps^j \chi_j \right) \left( \sum_{k=0}^\infty \eps^k \pi u_k \right) \\
        &\quad = \underbrace{u_0}_{\pi u_0} + \eps ( \underbrace{\chi_1 \pi u_0}_{\pi^\perp u_1} + \pi u_1 ) 
        + \eps^2 ( \underbrace{\chi_2 \pi u_0 + \chi_1 \pi u_1}_{\pi^\perp u_2} + \pi u_2 )
        + \eps^3 ( \underbrace{\chi_3 \pi u_0 + \chi_2 \pi u_1 + \chi_1 \pi u_2}_{\pi^\perp u_3} + \pi u_3 ) + \cdots.
    \end{split}
\end{equation}
For each $n\geq 0$, $\pi^\perp u_n$ is described in terms of $\pi u_0,\, \dots, \,\pi u_n$ and $\chi_j(y,\partial_t,\nabla_x)$, $0 \leq j \leq n-1$. The operators $\chi_j$ encode the $j$th-order (hyperbolic) correctors, and are defined inductively. We refer the reader to \cite[Definition~2.2]{allaire_lamacz_rauch_2022_crime_pays} for the precise definition of $\chi_j$ and \cite[Theorem~2.5]{allaire_lamacz_rauch_2022_crime_pays} for the statement of the equivalence of \eqref{eqn:choice_osc_classical} with \eqref{eqn:double_series}. The symbol $\chi_j(y,\cdot,\cdot)$ is a homogeneous polynomial of degree $j$, and its coefficients (as functions of $y$) belong to the space $\dot{H}^1_{\text{per}}(Y)$ (Definition \ref{defn:H1per_meanzero}). We have for instance, $\chi_1(y,\partial_t,\nabla_x) = \mathbf{N}(y)\cdot \nabla_x$.


\textbf{Step 4.} We turn to the non-oscillatory hierarchy ($\pi u_n$). Equation \eqref{eqn:wave_eqn_formal_v2} gives us 
\begin{align}
    \begin{cases}
        \mathcal{O}(\eps^{-2}): \quad \mathcal{A}_{yy} u_0 = w_{-2}, \\
        \mathcal{O}(\eps^{-1}): \quad \mathcal{A}_{yy} u_1 + \mathcal{A}_{xy} u_0 = w_{-1},
    \end{cases}
    \quad \text{which implies that} \quad
    \pi w_{-2} = 0 \quad \text{and} \quad \pi w_{-1} = 0. \label{eqn:nonosc_classical_base}
\end{align}
Moreover, it turns out the the choice \eqref{eqn:choice_osc_classical} forces an expression on the non-oscillatory parts of $w_n$, $n\geq 0$, as follows:

\begin{definition}
    Define the following constant coefficient operator of degree $n$:
    \begin{align}
        a_n^\ast\bigl(\partial_t,\nabla_x\bigr) \equiv a_n^\ast := \pi \bigl( (\partial_{tt} - \mathcal{A}_{xx})\chi_{n-2} - \mathcal{A}_{xy}\chi_{n-1} \bigr), \qquad n\geq 1.
    \end{align}
    We have for instance, $a_-^\ast = 0$ and $a_2^\ast = \partial_{tt} - \Div( \A^{\hom} \nabla)$ (the classical homogenised operator).
\end{definition}

Then, for $n\geq 0$, we have an equation relating $\pi w_n$ to $\pi u_0, \cdots, \pi u_n$  (cf. \eqref{eqn:u0_hom}):
\begin{alignat}{5}
    &\pi w_n 
    = \sum_{j=0}^n a_{j+2}^\ast \pi u_{n-j}
    &\text{By \cite[Theorem 2.10]{allaire_lamacz_rauch_2022_crime_pays}.} \label{eqn:crimepays_thm210}\\
    &=\begin{cases}
        \begin{alignedat}{6}
            a_2^\ast \pi u_n + 0 + a_4^\ast \pi u_{n-2} + \cdots 
            &+ a_{n}^\ast \pi u_2 &&+ 0 &&+ a_{n+2}^\ast \pi u_0 ~~~
            &&\text{if $n$ is even,} \\[0.1em]
            a_2^\ast \pi u_n + 0 + a_4^\ast \pi u_{n-2} + \cdots 
            &+ 0 &&+ a_{n+1}^\ast \pi u_1 &&+ 0
            &&\text{if $n$ is odd.} 
        \end{alignedat}
    \end{cases}
    &\text{By \cite[Theorem 2.13]{allaire_lamacz_rauch_2022_crime_pays}.} \label{eqn:crimepays_thm213}
\end{alignat}

\textbf{Step 5.} Finally, \textit{let us make the choice} (cf. system \eqref{eqn:classical_alg_system})
\begin{align}\label{eqn:nonosc_classical_wfirsttwo}
    \pi w_{0} = f, \qquad \text{and} \qquad \pi w_{n} = 0 \quad \text{for $n\geq 1$}.
\end{align}

Then, by equating like powers of $\eps$, we obtain a hierarchy of homogenised equations
\begin{align}\label{eqn:classical_alg_homo}
    \begin{cases}
        \mathcal{O}(1) &a_2^\ast \pi u_0 = f, \\
        \mathcal{O}(\eps^1) &a_2^\ast \pi u_1 = 0,\\
        \mathcal{O}(\eps^2) &a_2^\ast \pi u_2 = -a_4^\ast \pi u_0,\\
        \mathcal{O}(\eps^3) &a_2^\ast \pi u_3 = -a_4^\ast \pi u_1,\\
        \mathcal{O}(\eps^4) &a_2^\ast \pi u_4 = -a_4^\ast \pi u_2 - a_6^\ast \pi u_0,\\
        \mathcal{O}(\eps^5) &a_2^\ast \pi u_5 = -a_4^\ast \pi u_3 - a_6^\ast \pi u_1,\\
        \text{and so on...} 
    \end{cases}
\end{align}

The $\mathcal{O}(\eps)$ equation gives $\pi u_1 = 0$, which in turn gives $\pi u_{2k+1} = 0$ for all $k\geq 0$ (``leap-frog structure"). As for the terms $\pi u_{2k},$ we successively apply a standard energy estimate to obtain the following bound. 
\begin{theorem}\label{thm:secular_growth}
    \cite[Theorem~2.15]{allaire_lamacz_rauch_2022_crime_pays}
    For each non-zero multi-index $\alpha \in \N^{1+d}$ and $k\geq 0$, There exists $C=C(f,\alpha,k)>0$ such that
    \begin{align}
        \bigl\|\, \nabla_{t,x}^\alpha u_{2k}(t), \nabla_{t,x}^\alpha u_{2k+1}(t) \, \bigr\|_{L^2(\R_x^d \times \mathbb{T}_y^d)} \leq C \langle t \rangle^k, 
        \qquad \langle t \rangle := \sqrt{1+t^2}.
    \end{align}
\end{theorem}

The authors of \cite{allaire_lamacz_rauch_2022_crime_pays} refer to Theorem \ref{thm:secular_growth} as ``problems" of secular growth since it implies the following result.
\begin{theorem}\label{thm:classical_longtime}
    \cite[Theorem~3.1]{allaire_lamacz_rauch_2022_crime_pays}
    For each $k \in \mathbb{N} \cup \{ 0 \}$, define the truncated ansatz of level $k$ for $\sum \eps^n u_n$ by
    \begin{align}
        U^k(\eps,t,x,y) := \sum_{n=0}^{2k} \eps^n u_n(t,x,y) + \eps^{2k+1} \pi^\perp u_{2k+1} + \eps^{2k+2} \pi^\perp u_{2k+2}.
    \end{align}
    Then there is a constant $C = C(f,k)$ such that
    \begin{align}\label{eqn:classical_longtime}
        \left\| \nabla_{t,x} \left[ u_\eps(t,x) - U^k\biggl(\eps,t,x,\frac{x}{\eps}\biggr) \right] \right\|_{L^2(\R_x^d)} \leq C \min\Bigl\{ \eps^{2k+1} \langle t \rangle^{k+1}, \eps^{2k+2} \langle t \rangle^{k+2} \Bigr\}.
    \end{align}
\end{theorem}

The proof of Theorem \ref{thm:classical_longtime} is routine: one computes an explicit formula for the remainder and applies standard stability estimates. We observe that by taking $k$ large enough, the classical two-scale expansion $U^k$ provides a valid approximation of $u_\eps$ to arbitrary accuracy $\eps^\alpha$, but only up to times $t \leq \eps^{-2+\delta}$ (as $k\rightarrow \infty$, $\delta \downarrow 0$.), both in the energy norm $\| \nabla_{t,x} \cdot \|_{L^2(\R_x^d)}$ and in the $L^2(\R_x^d)-$norm.


It is crucial to point out that the $\mathcal{O}(\langle t \rangle^k)$ estimate in Theorem \ref{thm:secular_growth} directly trickles down to the factor $\langle t \rangle^k$ in the RHS of \eqref{eqn:classical_longtime}, and therein lies the ``problem": One does not have a uniform-in-$k$ control of the terms $u_n(t)$, and as a result, is restricted to times $t \leq \eps^{-2+\delta}$. Moreover, the timescale $t \sim \eps^{-2}$ is critical, as a one-dimensional example shows that $U^k(\eps,c/\eps^{2+\delta}, x, x/\eps)$ is unbounded in $x$ for fixed $k>0$, as $\eps\downarrow 0$, whereas the exact solution $u_\eps$ satisfies $\sup_{\eps} \| u_\eps \|_{L^\infty(\R_t \times \R_x)} < \infty$ \cite[Appendix A]{allaire_lamacz_rauch_2022_crime_pays}.

Returning to the comment after \eqref{eqn:spectral_germ_overall2} on the order-sharpness of \eqref{eqn:spectral_germ_overall2}, in case the reader would like to compare the error of \eqref{eqn:spectral_germ_overall2} with the $k=0$ case of \eqref{eqn:classical_longtime}, we point out that both guarantees a valid approximation up to times $t\leq \eps^{-1+\delta}$, but in different norms. Moreover the leading-order approximation $u_{\hom} = u_0$ in \eqref{eqn:spectral_germ_overall2} differs from the $0th-$order ansatz $U^0 = u_0 + \eps \chi_1 u_0 + \eps^2 \chi_2 u_0$.

\begin{remark}
    The leap-frog structure stems from \eqref{eqn:crimepays_thm213}, and is a unique feature of (scalar) WEs. For \textit{systems} of WEs, there is no leap-frog structure, and the secular growth of $u_k$ (as opposed to $u_{2k}$) in Theorem \ref{thm:secular_growth} becomes $\langle t \rangle^{k}$. The critical timescale is then $t\sim \eps^{-1}$.
\end{remark}

\subsection{Beyond the classical two-scale expansion}

It is now clear that to approximate $u_\eps$ up to times $t \sim \eps^{-2}$ and beyond, one has to leave the confines of the classical expansion. More precisely, a proposed ansatz has to address the secular growth problem. We shall outline the derivation of a few ansatze that overcame this problem.

\subsubsection{Criminal ansatz} In \cite{allaire_lamacz_rauch_2022_crime_pays}, Allaire, Lamacz-Keymling, and Rauch propose to seek an expansion of $u_\eps(t,x)$ of the form
\begin{align}\label{eqn:criminal_vseries}
    \sum_{n=0}^\infty \eps^n v_n, \qquad \text{where $v_n(\eps,t,x,y)$ is $\Z^d-$periodic in $y$.}
\end{align}
And under the equation \eqref{eqn:wave_eqn_formal_v2} (with $\sum \eps^n u_n$ replaced by $\sum \eps^n v_n$), make the following choices:
\begin{enumerate}[label=(\roman*)]
    \item \label{eqn:choice_osc_crmiminal} (Oscillatory hierarchy) Set $\pi^\perp w_n =0$ for all $n\geq -2$. (The same as \eqref{eqn:choice_osc_classical}.)
    \item \label{eqn:choice_nonosc_crmiminal} (Non-oscillatory hierarchy) Keep $\pi v_0$ as a free term. Impose two conditions: $\pi v_n = 0$ for $n\geq 1$, and 
    \begin{align}
        &\sum_{n=0}^\infty \eps^n (\pi w_n)(\eps,t,x) = f(t,x)
        &&\parbox{15em}{As opposed to \eqref{eqn:nonosc_classical_wfirsttwo}-\eqref{eqn:classical_alg_homo}. \\ Note that $\pi w_{n-2}=\pi w_{n-1}=0$ by \ref{eqn:choice_osc_crmiminal}.} \nonumber\\
        \Longleftrightarrow \quad
        &\sum_{n=0}^\infty \eps^n \sum_{j=0}^{n} a_{j+2}^\ast \pi v_{n-j} = f
        &&\text{By \cite[Theorem~2.10]{allaire_lamacz_rauch_2022_crime_pays}, cf. \eqref{eqn:crimepays_thm210}.} \nonumber\\
        \Longleftrightarrow \quad
        &\eps^0 \left( a_2^\ast \pi v_0 \right) 
        + \eps \left( a_2^\ast \pi v_1 + a_3^\ast \pi v_0 \right)
        + \eps^2 \left( a_2^\ast \pi v_2 + a_3^\ast \pi v_1 + a_4^\ast \pi v_0 \right)
        + \cdots = f.  \span\span \nonumber\\
        \Longleftrightarrow \quad
        & \left( a_2^\ast \pi v_0 \right)
        + \eps \left( a_2^\ast \pi v_1 \right)
        + \eps^2 \left( a_2^\ast \pi v_2 + a_4^\ast \pi v_0 \right)
        + \cdots = f.
        &&\parbox{15em}{By \cite[Theorem~2.13]{allaire_lamacz_rauch_2022_crime_pays}, \\ $a_k^\ast = 0$ for odd $k$, cf. \eqref{eqn:crimepays_thm213}.} \nonumber\\
        \Longleftrightarrow \quad
        &\left[ a_2^\ast(\partial_t,\nabla_x) + \eps^2 a_4^\ast(\partial_t,\nabla_x) + \eps^4 a_6^\ast(\partial_t,\nabla_x) + \cdots \right] \pi v_0 = f. \label{eqn:criminal_formal_homog_eqn_v1}
    \end{align}
    The final equivalence follows from our assumption that $\pi v_n = 0$, for $n\geq 1$.
\end{enumerate}

A few remarks are in order to motivate the choices \ref{eqn:choice_osc_crmiminal}-\ref{eqn:choice_nonosc_crmiminal}: By keeping the choice \ref{eqn:choice_osc_crmiminal} from the classical two-scale procedure, we retain the double-series/separation-of-variables structure \eqref{eqn:double_series} for our expansion for $u_\eps$, where the oscillatory terms $\pi^\perp v_n$ are expressed in terms of the corrector operators $\chi_j(y,\partial_t,\nabla_x)$ (defined through \textit{the same formulae as the classical procedure}) and the non-oscillatory terms $\pi v_n$. That is to say, since $\pi v_n = 0$ for $n\geq 1$, we have (formally) $\pi^\perp v_n = \chi_n \pi v_0$.

\begin{remark}
    In contrast, one has $\pi^\perp u_n = \chi_n \pi u_0 + \cdots + \chi_1 \pi u_{n-1}$ in the classical procedure.
\end{remark}

The authors refer to the dependence of $v_n$ on $\eps$, and the condition \eqref{eqn:criminal_formal_homog_eqn_v1} (obtained by mixing equations of different powers of $\eps$) as ``asymptotic crimes".

Due to such criminal acts, one has to take care to interpret the formal ``homogenised equation" $\eqref{eqn:criminal_formal_homog_eqn_v1}$. Indeed, a simple truncation of the candidate equation \eqref{eqn:criminal_formal_homog_eqn_v1} typically leads to ill-posed problems. Note also, that by allowing $v_n$'s to depend on $\eps$, we should expect $v_n$ to differ between truncation levels. That is, with truncations $\sum_{n=0}^N \eps^n v_n^N$ and $\sum_{n=0}^M \eps^n v_n^M$, we generally have $v_n^N \neq v_n^M$, on the contrary to the law-abiding classical two-scale expansion.

We shall now briefly describe the steps taken by the authors to turn \eqref{eqn:criminal_formal_homog_eqn_v1} into a well-posed problem:

\begin{enumerate}
    \item \label{criminal_normal_form} \textbf{Algebraic step: Normal-form transformation.} Keep $a_2^\ast(\partial_t,\nabla_x)$ as is. Remove the $\partial_t$ from $a_4^\ast$, $a_6^\ast$, $\cdots$, through an ``elimination algorithm" on \eqref{eqn:criminal_formal_homog_eqn_v1}: There exists \textit{uniquely determined} homogeneous operators $R_{2j}(\partial_t, \nabla_x)$ and $\widetilde{a}_{2j}(\nabla_x)$ of degree $2j$, such that as a formal series \cite[Proposition\,4.1]{allaire_lamacz_rauch_2022_crime_pays},
    \begin{align}\label{eqn:elimination_alg}
        a_2^\ast(\partial_t,\nabla_x) + \sum_{j=2}^\infty \widetilde{a}_{2j}(\nabla_x) = \left[ I + \sum_{j=1}^\infty R_{2j}(\partial_t, \nabla_x) \right] \left[ \sum_{j=1}^\infty a^\ast_{2j}(\partial_t,\nabla_x) \right].
    \end{align}
    Homogeneity implies for instance, that $R_{2j}(\eps\partial_t,\eps\nabla_x) = \eps^{2j}R_{2j}(\partial_t,\nabla_x)$. Thus, by multiplying \eqref{eqn:criminal_formal_homog_eqn_v1} on both sides on the left by $I+\sum R_{2j}(\eps \partial_t, \eps \nabla_x)$, we have, by \eqref{eqn:elimination_alg},
    \begin{align}\label{eqn:criminal_formal_homog_eqn_v2}
        \left[ a_2^\ast(\partial_t,\nabla_x) + \sum_{j=2}^\infty \eps^{2j-2} \widetilde{a}_{2j}(\nabla_x) \right] \pi v_0
        = \left[ I+\sum_{j=1}^\infty \eps^{2j} R_{2j}(\partial_t, \nabla_x) \right] f.
    \end{align}
    
    We have thus successfully ``de-mixed" the space and time derivatives on the LHS of \eqref{eqn:criminal_formal_homog_eqn_v1}, at the expense of (slightly) modifying the RHS.

    \item \label{criminal_filtering} \textbf{Analytic step: Filtering.} Apply $\psi(\eps^{\alpha}(-i)\nabla_x)$ to the RHS of \eqref{eqn:criminal_formal_homog_eqn_v2}, with fixed $0<\alpha<1$ and $\psi \in C^\infty_c(\R^d)$ with $\psi = 1$ in the neighborhood of the origin. The resulting equation at truncation level $k$ is
    \begin{equation}\label{eqn:criminal_formal_homog_eqn_v3}
        \begin{split}
            &\left[ a_2^\ast(\partial_t,\nabla_x) + \eps^2 \widetilde{a}_4(\nabla_x) + \cdots + \eps^{2k} \widetilde{a}_{2k+2}(\nabla_x) \right] \pi v_0^k
            = \psi(\eps^{\alpha}(-i)\nabla_x) \left[ I + \eps^2 R_{2} + \cdots + \eps^{2k} R_{2k} \right] f.
        \end{split}
    \end{equation}
    
\end{enumerate}

Thus, at truncation level $k$, we have a uniquely-defined non-oscillatory profile $\pi v_0^k$ by \eqref{eqn:criminal_formal_homog_eqn_v3}. All other non-oscillatory profiles $\pi v_1^k, \pi v_2^k,\cdots$ are set to $0$. The oscillatory terms are simply $\pi^\perp v_0^k = 0$ and $\pi^\perp v_n^k = \chi_n v_0^k$ for $n\geq 1$. The proposed ansatz (at level $k$) is
\begin{align}\label{eqn:criminal_ansatz}
    V^k(\eps,t,x,y) := \left[ I + \eps \chi_1(y,\partial_t,\nabla_x) + \cdots \eps^{2k+2} \chi_{2k+2}(y,\partial_t,\nabla_x) \right] v_0^k(\eps,t,x).
\end{align}

It turns out that the ``criminal ansatz" \eqref{eqn:criminal_ansatz} provides a description of $u_\eps$ that is good to an arbitrary order of accuarcy $\eps$ and timescale $t$:

\begin{theorem}\label{thm:criminal_longtime}
    \cite[Theorem~1.3 and Corollary~1.7]{allaire_lamacz_rauch_2022_crime_pays}
    For each $k \in \mathbb{N} \cup \{ 0 \}$, there exist $C(f,k)>0$ such that
    \begin{align}\label{eqn:criminal_longtime}
        \left\| \nabla_{t,x} \left[ u_\eps(t,x) - V^k(\eps,t,x,\tfrac{x}{\eps}) \right] \right\|_{L^2(\R_x^d)} \leq C \eps^{2k+1} \langle t \rangle,
    \end{align}
    and 
    \begin{align}\label{eqn:criminal_longtime_v0only}
        \left\| \nabla_{t,x} \left[ u_\eps(t,x) - v_0^k(\eps,t,x) \right] \right\|_{L^2(\R_x^d)} \leq C \left( \eps + \eps^{2k+2} \langle t \rangle^2 \right).
    \end{align}
\end{theorem}

How to use Theorem \ref{thm:criminal_longtime}: For any $N,M\geq 1$, if one desires an approximation (in the energy norm) with accuracy $\eps^{N}$ that is valid on times $|t| \leq C\eps^{-M}$, then one can take $V^k$ with any $k$ satisfying $N+M \leq 2k+2$. Keeping only the non-oscillatory profile $v_0^k(\eps,t,x)$ still gives a valid approximation to arbitrary long timescales, but with a maximum $\mathcal{O}(\eps)$ accuracy.

\paragraph*{Summary.} The criminal ansatz $V^k(\eps,t,x,y) = \sum_{n=0}^k \eps^n v_n(\eps,t,x,y)$ \eqref{eqn:criminal_ansatz} is an asymptotic expansion for $u_\eps$ that retains the double-series structure \eqref{eqn:double_series} of the classical two-scale expansion, with correctors $\chi_j(y,\partial_t,\nabla_x)$ defined in terms of the non-oscillatory terms $\pi v_n$ in the same way. Secular growth is avoided by replacing \eqref{eqn:nonosc_classical_wfirsttwo}-\eqref{eqn:classical_alg_homo} with \ref{eqn:choice_nonosc_crmiminal}, resulting in a valid approximation (in the energy norm) to arbitrary long timescales, by taking $k$ large enough.

\subsubsection{Interlude} 
\paragraph*{A discussion of the literature.} We now make some comments connecting the criminal ansatz \cite{allaire_lamacz_rauch_2022_crime_pays} to the wider literature. It was first observed numerically by Santosa and Symes \cite{santosa_symes_1991} that the classical homogenised description $u_{\hom}$ is inaccurate at times $O(\eps^{-2})$ due to the presence of dispersion at such timescales. To counteract this, the authors proposed an ansatz that is good to $O(\eps^{-2})$ in time, describing a weakly dispersive effective medium, and does not follow the two-scale expansion recipe. The validity of this ansatz was first proven by Lamacz in the one-dimensional setting \cite{lamacz_1D}, and then extended to dimensions $d\leq 3$ by Dohnal, Lamacz, and Schweizer \cite{dohnal_lamacz_schweizer_2014}. \footnote{We point out that the works \cite{lamacz_1D, dohnal_lamacz_schweizer_2014} differ slightly from the setting discussed here (cf.\,Assumption \ref{assumption}). For \cite{lamacz_1D}, $f=0$, with $u_\text{init}$ and $v_\text{init}$ smooth. For \cite{dohnal_lamacz_schweizer_2014}, $f=0$, $v_\text{init}=0$, and $u_\text{init}$ smooth.} The ansatz $w_\eps$ in \cite{dohnal_lamacz_schweizer_2014} solves the well-posed equation
\begin{align}
    \underbrace{\partial_{tt} w_\eps - \mathbb{A}^{\hom} : \nabla^2 w_\eps}_{a_2^\ast(\partial_t, \nabla_x) w_\eps} = \eps^2 \mathbb{E} : \nabla^2  \partial_{tt} w_\eps - \eps^2 \mathbb{F} \odot \nabla^4 w_\eps,
    \qquad w_\eps(\cdot,0) = u_\text{init},
    \qquad \partial_t w_\eps(\cdot,0) = 0.
\end{align}
$w_\eps$ is valid approximation of $u_\eps$ on times $t\leq C\eps^{-2}$ with accuracy $\mathcal{O}(\eps)$ in the $\| \cdot \|_{L^2+L^\infty}$ norm. Here, $\mathbb{E} \in \R^{d\times d}$ and $\mathbb{F} \in \R^{d\times d\times d \times d}$ are non-negative (constant) tensors. They describe the ``weakly-dispersive" effects, and are extracted from the Bloch-wave expansion of $u_\eps$\footnote{Actually, $\mathbb{E}$ and $\mathbb{F}$ have to be suitably modified from Bloch-data to prevent ill-posed issues. This is the analogue of \ref{criminal_normal_form}-\ref{criminal_filtering} in the criminal ansatz. In \cite{dohnal_lamacz_schweizer_2014}, the so-called ``Boussinesq trick" was used. We discuss this below.}, similar to the first step of the spectral germ approach.

Under the class of Bloch-wave/spectral methods, we note the development of an ``approximate Floquet theory" by Benoit and Gloria \cite{benoit_gloria_2019_ballistic_transport}, which is applicable to the stochastic setting. In the context of the (deterministic) WE, \cite{benoit_gloria_2019_ballistic_transport} provides an (spectral) ansatz that is valid for arbitrarily long times, but with maximum accuracy $\mathcal{O}(\eps)$ (cf. \eqref{eqn:criminal_longtime_v0only}). Bridging the gap from \eqref{eqn:criminal_longtime_v0only} to \eqref{eqn:criminal_longtime} is the content of Duerinckx, Gloria, and Ruf \cite{duerinckx_gloria_ruf_2024_spectral_ansatz}, which we shall discuss below.

The work \cite{allaire_lamacz_rauch_2022_crime_pays} provides the first rigorous justification that $t\sim \eps^{-2}$ is the critical timescale for the classical two-scale expansion. Moreover, it is the first work that provides an ansatz $V^k$ that is simultaneously arbitrarily accurate and valid for arbitrarily long times, for all dimensions.

\paragraph*{Connection between physical and frequency space.} Let us explain the connection between the two-scale homogenised data (homogenised operators $a_n^\ast(\partial_t,\nabla_x)$ and correctors $\chi_j(y,\partial_t,\nabla_x)$), and the Bloch data (the spectral information of $\mathcal{A}_\chi \equiv (\nabla_y + i\chi)^\ast \A(y) (\nabla_y + i\chi)$) \cite{zhikov1989}. Note that $\mathcal{A}_\chi$ has discrete spectrum.

We focus on the first eigenpair $(\lambda_-^\chi, \varphi_1^\chi)$ of $\mathcal{A}_\chi$, for small $\chi$. By the min-max principle \cite[Theorem~5.15]{david_borthwick}, the first eigenvalue $\lambda_-^0=0$ of $\mathcal{A}_{\chi=0}$ is simple and isolated. Then perturbation theory \cite{kato_book} applies, and the first band function $\chi \mapsto \lambda_-^\chi$ is analytic in a neighbourhood of $\chi=0$. Moreover, $\lambda_-^{(\cdot)}$ is even, as $(\lambda^\chi, \phi^\chi)$ is an eigenpair for $\mathcal{A}_\chi$ if $(\lambda^\chi, \overline{\phi^\chi})$ is an eigenpair for $\mathcal{A}_{-\chi}$. Thus $\lambda_-^{(\cdot)}$ admits the following Taylor expansion about $\chi=0$:
\begin{align}\label{eqn:first_bloch_evalue}
    \lambda_-^\chi = \mathbb{A}^0 \chi \cdot \chi + \mathcal{O}(|\chi|^4), \qquad \text{for some $\mathbb{A}^0 \in \R_{\text{sym}}^{d\times d}$.}
\end{align}
Write $P_1^\chi$ for the projection of $L^2(Y)$ onto the first eigenspace of $\mathcal{A}_\chi$. By perturbation theory, $\chi\mapsto P_1^\chi$ is also analytic in a neighbourhood of $\chi=0$. Thus we may write
\begin{align}\label{eqn:first_bloch_efct}
    \varphi_1^\chi 
    = P_1^\chi \underbrace{1}_{\text{Remark \ref{rmk:ker_A0}: $P_1^0 = P_{\C}$}}
    = 1 + \mathbf{M}^{(1)} \cdot \chi + \mathbf{M}^{(2)} : \chi \otimes \chi + \mathcal{O}(|\chi|^3),
\end{align}
where $\mathbf{M}^{(1)} = (M_1,\cdots,M_d)^\top$, $\mathbf{M}^{(2)} = (M_{ij})_{1\leq i,j \leq d}$, and $M_i$, $M_{ij}$ are in $L^2(Y)$. Now substitute the expansions \eqref{eqn:first_bloch_evalue}-\eqref{eqn:first_bloch_efct} into the eigenvalue equation 
\begin{align}\label{eqn:evalue_eqn_expand}
    &-\nabla_y \cdot \left( \A(y) \nabla_y \varphi_1^\chi \right)
    - (i\chi)^\ast \A(y) \nabla_y \varphi_1^\chi 
    - \nabla_y \cdot \left( \A(y) (i\chi) \varphi_1^\chi \right)
    - (i\chi)^\ast \A(y) (i\chi) \varphi_1^\chi
    = \lambda_-^\chi \varphi_1^\chi,
\end{align}
and obtain a system of equations by equating like powers of $|\chi|$. Going up to $\mathcal{O}(|\chi|^2)$, one recovers the homogenised matrix and the first-order corrector! To be precise, we have
\begin{align}\label{eqn:bloch_connection_chi2}
    \lambda_-^\chi = \A^{\hom}\chi \cdot \chi + \mathcal{O}(|\chi|^4),
    \qquad \text{and} \qquad
    \varphi_1^\chi = 1 + i\mathbf{N} \cdot \chi + \mathcal{O}(|\chi|^2).
\end{align}

We refer the reader to the survey of Zhikov and Pastukova \cite[Sect~9 and 12.2]{zhikov_pastukhova_2016_opsurvey} for details.

The above procedure suggests that Taylor coefficients of $(\lambda_-^\chi, \varphi_1^\chi)$ could encode higher-order two-scale homogenised data. Indeed, going up to $\mathcal{O}(|\chi|^4)$, Conca, Orive, and Vanninathan showed that \cite{conca_orive_vanninathan2002}
\begin{theorem}\label{thm:connection_evalue}
    \cite[Proposition~1.9]{conca_orive_vanninathan2002} Consider the Taylor expansions of the first band function $\chi \mapsto \lambda_-^\chi$ near $\chi=0$. All odd-derivatives of $\lambda_-^{(\cdot)}$ vanish. Moreover, we have a characterization of the second and fourth Taylor coefficients in terms of two-scale homogenised data: 
    \begin{align}\label{eqn:first_bloch_evalue_expansion}
        \lambda_-^\chi = \mathbb{A}^{\hom} \chi \cdot \chi + \mathbb{D} \odot (\chi \otimes \chi \otimes \chi \otimes \chi) + \mathcal{O}(|\chi|^6),
    \end{align}
    where $\mathbb{A}^{\hom}$ is the homogenised matrix in Theorem \ref{thm:basic_wave_homog}, and the \textit{Brunett tensor} $\mathbb{D} \in \R^{d\times d \times d\times d}$ is defined by
    \begin{align}\label{eqn:brunett_tensor}
        \mathbb{D} := - \int_Y \left( \A(y) \otimes \mathbf{N}^{(2)}(y) + (\A(y)\nabla_y) \otimes \widehat{\mathbf{M}}^{(3)}(y) \right) dy,
    \end{align}
    where $\widehat{\mathbf{M}}^{(3)} = (M_{ijk})_{1\leq i,j,k \leq d}$ solves the cell-problem
    \begin{align}
        \mathcal{A}_{yy} \widehat{\mathbf{M}}^{(3)} = \mathbb{A}^{\hom} \otimes \mathbf{N}^{(1)} + \mathcal{A}_{yy} \mathbf{N}^{(3)}, \qquad M_{ijk} \in \dot{H}^1_{\text{per}}(Y).
    \end{align}
    We have written $\mathbf{N}^{(j)}$ for the $j$th-order correctors from the two-scale expansion.
\end{theorem}

Moreover, the same authors showed that the Brunett tensor $\mathbb{D}$ is non-positive on rank-one matrices:

\begin{proposition}\label{prop:brunett_sign}
    \cite[Section~B]{conca_orive_vanninathan2006} For all $\xi \in \R^d$, we have $\mathbb{D}(\xi\otimes \xi) : (\xi\otimes \xi) \leq 0$.
\end{proposition}
This should be contrasted with the fact that $\mathbb{A}^{\hom}$ is positive (due to the ellipticity assumption on $\mathbb{A}(y)$). As for the ground state $\varphi_1^\chi$, the same authors showed that
\begin{theorem}\label{thm:connection_efct}
    \cite[Proposition~1.10]{conca_orive_vanninathan2002}\footnote{The authors of \cite{conca_orive_vanninathan2002} actually computed the $\mathcal{O}(|\chi|^4)$ term, but we have omitted the formulae to streamline the discussion.} The Taylor expansion of $\chi \mapsto \varphi_1^\chi$ near $\chi=0$ is
    \begin{equation}\label{eqn:first_bloch_efct_expansion}
    \begin{split}
        \varphi_1^\chi(y) &= 
        1 + \mathbf{N}^{(1)}(y) \cdot (i \chi) 
        + \left[ 2 \mathbf{N}^{(2)}(y) 
        + \int_Y (\mathbf{N}^{(1)} \otimes \mathbf{N}^{(1)}) \, d\tilde{y} \right] : \frac{1}{2!} (i \chi) \otimes (i \chi) \\
        &+ \bigg[ \widehat{\mathbf{M}}^{(3)}(y) - \frac{1}{3} \bigg( 
        \mathbf{N}^{(1)}(y) \otimes \int_Y (\mathbf{N}^{(1)} \otimes \mathbf{N}^{(1)}) \, d\tilde{y} \\
        &\qquad+ \int_Y \mathbf{N}^{(1)}(\tilde{y}) \otimes \mathbf{N}^{(1)}(y) \otimes \mathbf{N}^{(1)} (\tilde{y}) \, d\tilde{y}
        + \int_Y (\mathbf{N}^{(1)} \otimes \mathbf{N}^{(1)}) \, d\tilde{y} \otimes \mathbf{N}^{(1)}(y)
        \bigg) 
        \bigg] \odot \frac{1}{3!} (i\chi)^{\otimes 3} + \mathcal{O}(|\chi|^4).\\
    \end{split}
    \end{equation}
\end{theorem}

It was remarked in \cite{conca_orive_vanninathan2002}, that one could in-principle carry out similar computations to connect Bloch data and two-scale data to all orders, but this line of investigation remains open.


As we have seen above, the leading-order Taylor coefficients of $(\lambda_-^\chi, \varphi_1^\chi)$ coincide with the two-scale homogenised data $\mathbb{A}^{\hom}$ and $\mathbf{N}^{(0)}=1$, but this is not true for higher-orders. For instance, the Brunett tensor $\mathbb{D}$ differs from the fourth-order (stationary) two-scale homogenised coefficient
\begin{align}\label{eqn:b_hom}
    \mathbb{B}^{\hom} := \int_Y \left( \A(y) \otimes \mathbf{N}^{(2)}(y) + (\A(y)\nabla_y) \otimes \mathbf{N}^{(3)}(y) \right) \, dy.
\end{align}
That is, $\pi u_2$ solves $\partial_{tt} \pi u_2 -\Div ( \A^{\hom} \nabla \pi u_2 ) = (\mathbb{B}^{\hom}:\nabla^4 + \text{mixed derivatives}\,)\, \pi u_0$ (see \eqref{eqn:classical_alg_homo}). Nonetheless, the formulae \eqref{eqn:b_hom} and \eqref{eqn:brunett_tensor} are close enough, that one could ask if there are conditions such that $\mathbb{D}$ and $\mathbb{B}^{\hom}$ can be made to coincide. This is the content of \cite{allaire_briane_vanninathan_comparison}: Allaire, Briane, and Vanninathan showed that if the forcing term $f(t,x)$ is suitably modified, then the $\mathcal{O}(\eps^2)$-homogenised WE from \eqref{eqn:classical_alg_homo} coincides with the equation formally obtained by applying the Fourier transform to the Bloch data:\footnote{Again, the equation \eqref{eqn:abv_homogenised_we} is ill-posed, and a suitable modification is necessary. We discuss this below.}
\begin{align}\label{eqn:abv_homogenised_we}
    \partial_{tt} v_\eps - \Div (\A^{\hom} \nabla v_\eps ) + \eps^2 \mathbb{D} \odot \nabla^4 v_\eps = f.
\end{align}
We refer the reader to \cite[Proposition~6.1]{allaire_briane_vanninathan_comparison} for the precise statement.

\subsubsection{Spectral ansatz} We shall summarize the derivation of the ansatz proposed by Benoit-Gloria-Duerinckx-Ruf \cite{benoit_gloria_2019_ballistic_transport, duerinckx_gloria_ruf_2024_spectral_ansatz}. \textbf{Step 1.} Following \eqref{eqn:evalue_eqn_expand}, let consider for each frequency/wave-vector $\xi \in \R^d$, the operator 
\begin{align}
    \mathcal{A}_{\xi} 
    &= -(\nabla_y + i\xi) \cdot \A(y) (\nabla_y + i\xi) \nonumber \\
    &= \underbrace{- \nabla_y \cdot \A(y) \nabla_y}_{=:~ \mathcal{L}_\xi^{(0)}} 
    ~ \underbrace{- \nabla_y \cdot (\A(y) (i\xi) ) - (i\xi) \cdot \A(y) \nabla_y}_{=: ~\mathcal{L}_\xi^{(1)}} 
    ~ \underbrace{- (i\xi) \cdot \A(y) (i\xi) }_{=: ~\mathcal{L}_\xi^{(2)}}
\end{align}
on $L^2(Y)$, equipped with periodic boundary conditions. The superscript in $\mathcal{L}^{(j)}$ loosely indicates that there are $j$-factors of $(i\xi)$'s. Also, the reader should compare the operators $\mathcal{L}^{(0)}$, $\mathcal{L}^{(1)}$, and $\mathcal{L}^{(2)}$ with $\mathcal{A}_{yy}$, $\mathcal{A}_{xy}$, and $\mathcal{A}_{xx}$ of the two-scale expansion respectively, formally replacing $(i\xi)$ with $\nabla_x$.

\textbf{Step 2.} We shall now set $\xi = \eps \chi$, where $\chi \in Y' = [-\pi,\pi]^d$. Then,
\begin{align}\label{eqn:spec_ansatz_operator_expand}
    \mathcal{A}_{\eps\chi} 
    &= - \nabla_y \cdot \A(y) \nabla_y - \eps \nabla_y \cdot (\A(y) (i\chi) ) - \eps (i\chi) \cdot \A(y) \nabla_y - \eps^2 (i\chi) \cdot \A(y) (i\chi) \nonumber\\
    &= \mathcal{L}_\chi^{(0)} + \eps \mathcal{L}_\chi^{(1)} + \eps^2 \mathcal{L}_\chi^{(2)}.
\end{align}
We are interested in the eigenvalue equation with $\xi = \eps \chi$ for the first eigenpair:
\begin{align}\label{eqn:spec_anzatz_evalue_eqn}
    \mathcal{A}_{\eps \chi} \varphi_1^{\eps\chi} = \lambda_-^{\eps\chi} \varphi_1^{\eps\chi}.
\end{align}

\begin{remark}
    [Comparing to the previous section] We have sought an expansion of $(\lambda_-^{\chi},\varphi_1^{\chi})$ near $\chi = 0$ in \eqref{eqn:first_bloch_evalue_expansion}-\eqref{eqn:first_bloch_efct_expansion}. In the present setup, we shall expand near $\eps\chi = \xi = 0$. However, we do not seek an expansion in the variable $\xi$, but instead will introduce a two-stage expansion: first in $\eps$, then in $\chi$.
\end{remark}

Recall that when $\xi=\eps\chi$ is small, $\lambda_-^{\xi}$ is simple and standard perturbation theory applies \cite{kato_book}. Thus, let us seek an expansion for $\lambda_-^{\eps \chi}$ and $\varphi_1^{\eps \chi}$, in powers of $\eps$:
\begin{align}\label{eqn:spec_ansatz_eigenpair_expand}
    \lambda_-^{\eps \chi} \sim \sum_{n=0}^\infty \eps^n \check{\lambda}_\chi^{(n)},
    \quad \text{where $\check{\lambda}_\chi^{(n)}\in \C$,}
    \qquad \text{and} \qquad
    \varphi_1^{\eps \chi} \sim \sum_{n=0}^\infty \eps^n \check{\varphi}_\chi^{(n)},
    \quad \text{where $\check{\varphi}_\chi^{(n)}$ is $\Z^d-$periodic.}
\end{align}

\textbf{Step 3.} Substitute \eqref{eqn:spec_ansatz_operator_expand} and \eqref{eqn:spec_ansatz_eigenpair_expand} into the eigenvalue equation \eqref{eqn:spec_anzatz_evalue_eqn}. Collect like powers of $\eps$:
\begin{equation}\label{eqn:spec_ansatz_hierachy}
    \begin{cases}
        \begin{alignedat}{6}
            &\mathcal{O}(\eps^0) &&\mathcal{L}_\chi^{(0)} \check{\varphi}_\chi^{(0)} 
            &&= \check{\lambda}_\chi^{(0)} \check{\varphi}_\chi^{(0)}, \\
            &\mathcal{O}(\eps^1) &&\mathcal{L}_\chi^{(0)} \check{\varphi}_\chi^{(1)} + \mathcal{L}_\chi^{(1)} \check{\varphi}_\chi^{(0)} 
            &&= \check{\lambda}_\chi^{(0)} \check{\varphi}_\chi^{(1)} + \check{\lambda}_\chi^{(1)} \check{\varphi}_\chi^{(0)}, \\
            &\mathcal{O}(\eps^2) &&\mathcal{L}_\chi^{(0)} \check{\varphi}_\chi^{(2)} + \mathcal{L}_\chi^{(1)} \check{\varphi}_\chi^{(1)} + \mathcal{L}_\chi^{(2)} \check{\varphi}_\chi^{(0)} 
            &&= \check{\lambda}_\chi^{(0)} \check{\varphi}_\chi^{(2)} + \check{\lambda}_\chi^{(1)} \check{\varphi}_\chi^{(1)} + \check{\lambda}_\chi^{(2)} \check{\varphi}_\chi^{(0)}, \\
            &~~\vdots \qquad\qquad && &&\\
            &\mathcal{O}(\eps^k) &&\mathcal{L}_\chi^{(0)} \check{\varphi}_\chi^{(k)} 
            &&= - \mathcal{L}_\chi^{(1)} \check{\varphi}_\chi^{(k-1)}
            - \mathcal{L}_\chi^{(2)} \check{\varphi}_\chi^{(k-2)} 
            + \sum_{n=0}^k \check{\lambda}_\chi^{(n)} \check{\varphi}_\chi^{(k-n)}.
        \end{alignedat}
    \end{cases}
\end{equation}
Since we know that $\lambda_-^{0}=0$ and $\varphi_1^{0} = 1$, it is thus natural to set $\check{\lambda}_\chi^{(0)} := 0$ and $\check{\varphi}_\chi^{(0)} := 1$.

\textbf{Step 4.} We shall write down the equations for $\check{\lambda}_\chi^{(k)}$ and $\check{\varphi}_\chi^{(k)}$, $k\geq 1$. For $\check{\varphi}_\chi^{(k)}$, \eqref{eqn:spec_ansatz_hierachy} reads:
\begin{align}\label{eqn:spec_ansatz_hierachy_k}
    - \nabla_y \cdot \left( \A(y) \nabla_y \check{\varphi}_\chi^{(k)} \right)
    = \nabla_y \cdot \left( \A(y) (i\chi) \check{\varphi}_\chi^{(k-1)} \right)
    + (i\chi) \cdot \A(y) \left( \nabla_y \check{\varphi}_\chi^{(k-1)} + (i\chi) \check{\varphi}_\chi^{(k-2)} \right)
    + \sum_{n=0}^k \check{\lambda}_\chi^{(n)} \check{\varphi}_\chi^{(k-n)},
\end{align}
where we shall impose the mean-zero condition $\int_Y \check{\varphi}_\chi^{(k)} \, dy = 0$, as we typically do for all cell-problems.

For $\check{\lambda}_\chi^{(k)}$, observe that by Divergence theorem and periodicity of $\check{\varphi}_\chi^{(k)}$,
\begin{align}
    \int_Y \nabla_y \cdot \left( \A(y) \nabla_y \check{\varphi}_\chi^{(k)} \right) \, dy = 0,
    \qquad \text{and} \qquad
    \int_Y \nabla_y \cdot \left( \A(y) (i\chi) \check{\varphi}_\chi^{(k-1)} \right) \, dy = 0.
\end{align}
Also, due to the periodicity of $\check{\varphi}_\chi^{(n)}$'s, the condition $\int_Y \check{\varphi}_\chi^{(n)} \, dy = 0$ ($n\geq 1$), and $\check{\varphi}_\chi^{(0)}=1$, we get 
\begin{align}
    \int_Y \sum_{n=0}^k \check{\lambda}_\chi^{(n)} \check{\varphi}_\chi^{(k-n)} \, dy
    = \sum_{n=0}^k \check{\lambda}_\chi^{(n)} \int_Y \check{\varphi}_\chi^{(n)} \, dy
    = \check{\lambda}_\chi^{(k)}.
\end{align}
Thus, by taking $\int_Y dy$ in \eqref{eqn:spec_ansatz_hierachy_k}, one arrives at the equation for $\check{\lambda}_\chi^{(k)}$:
\begin{align}
    \check{\lambda}_\chi^{(k)} = - \int_Y (i\chi) \cdot \A(y) \left( \nabla_y \check{\varphi}_\chi^{(k-1)} + (i\chi) \check{\varphi}_\chi^{(k-2)} \right) \,dy
\end{align}

\textbf{Step 5.} The above steps are sufficient, if one is content with a maximum accuracy of $\mathcal{O}(\eps)$. To obtain an ansatz that is valid for arbitrarily long times and to arbitrary accuracy \textit{simultaneously}, one needs to perform an expansion for the ``bulk" (the remaining eigenspaces). However, the observation made in \cite{duerinckx_gloria_ruf_2024_spectral_ansatz} is that one do not need an expansion for the individual eigenprojections $P_2^{\chi}$, $P_3^{\chi}$, $\cdots$. Rather, an expansion for the \textit{sum} $P_2^{\chi} + P_3^\chi + \cdots = (P_1^\chi)^\perp$ suffices.\footnote{This observation was also used in a different manner by Cherednichenko-Velčić-Žubrinić-Lim to develop an ``operator-asymptotic" approach to homogenisation \cite{simplified_method, kirill_igor_josip_rods, kirill_igor_plates}.} To this end, consider the function
\begin{align}
    \Psi_{\chi,\eps}^m := (\mathcal{A}_{\eps\chi})^{-m-1} \tfrac{1}{\eps} (P_1^{\eps\chi})^\perp 1, \qquad \text{for $m\geq 0$.}
\end{align}
We remark the the powers $(\mathcal{L}_{\eps\chi})^{-m-1}$ arises naturally from the expansion of the operator $\sin(\mathcal{L}_{\eps\chi}^{1/2}(t-s))$ in the Duhamel formula \eqref{hyperbolic_formula}. For each fixed $m\geq 0$, seek an expansion for $\Psi_{\chi,\eps}^m$ in powers of $\eps$:
\begin{align}\label{eqn:spec_ansatz_bulk_expand}
    \Psi_{\chi,\eps}^m \sim \frac{1}{\| \varphi_1^{\eps\chi} \|} \sum_{n=0}^\infty \eps^n \check{\zeta}_\chi^{(n,m)},
    \quad \text{where $\check{\varphi}_\chi^{(n)}$ is $\Z^d-$periodic.}
\end{align}

\textbf{Step 6.} Substituting the expansion \eqref{eqn:spec_ansatz_bulk_expand} into the equation
\begin{align}
    \mathcal{A}_{\eps\chi} \Psi_{\chi,\eps}^0 = \tfrac{1}{\eps} (P_1^{\eps\chi})^\perp 1,
\end{align}
one obtains a hierarchy of equations for $\check{\zeta}_\chi^{(n,0)}$ for each $n\geq 0$ (we omit this for brevity, see \cite[Sect~1.4]{duerinckx_gloria_ruf_2024_spectral_ansatz} for details). Then, using the relation
\begin{align}
    \mathcal{A}_{\eps\chi} \Psi_{\chi,\eps}^m = \Psi_{\chi,\eps}^{m-1}, \qquad m\geq 1,
\end{align}
one obtains a hierarchy of equations for $\check{\zeta}_\chi^{(n,m)}$ for each $n\geq 0$ and $m \geq 1$:
\begin{align}
    -\nabla_y \cdot \A(y) \nabla_y \check{\zeta}_\chi^{(n,m)} = 
    \nabla_y \cdot \left( \A(y) (i\chi) \check{\zeta}_\chi^{(n-1,m)} \right)
    + (i\chi) \cdot \A(y) \left( \nabla_y \check{\zeta}_\chi^{(n-1,m)} + (i\chi) \check{\zeta}_\chi^{(n-2,m)} \right)
    + \check{\zeta}_\chi^{(n,m-1)},
\end{align}
where we pick a convenient choice of $\int_Y \check{\zeta}_\chi^{(n,m)} \, dy$ so that they are uniquely defined.

\textbf{Step 7.} Extract the ``Bloch data" $\A^{\hom,n}$, $\varphi^{(n)}$, and $\zeta^{(n,m)}$ by expanding in powers of $(i\chi)^{\otimes n}$, $n\geq 0$:
\begin{align}
    \check{\lambda}_\chi^{(n+1)} = \chi \cdot (\A^{\hom,n} \odot (i\chi)^{\otimes (n-1)}    )\chi,
    \qquad
    \check{\varphi}_\chi^{(n)} = \varphi^{(n)} \odot (i\chi)^{\otimes n},
    \qquad
    \check{\zeta}_\chi^{(n,m)} = \zeta^{(n,m)} \odot (i\chi)^{\otimes (n+1)}.
\end{align}
For instance, when $k=1$, we have $\A^{\hom,1} = \A^{\hom}$, and so $\check{\lambda}_\chi^{(2)} = \chi \cdot \A^{\hom} \chi$ (cf.~\eqref{eqn:bloch_connection_chi2}), and the $\mathcal{O}(\eps)$ equation of \eqref{eqn:spec_ansatz_hierachy} reads
\begin{align}
    \mathcal{L}_\chi^{(0)} \varphi^{(1)} \cdot (i\chi) = \nabla_y \cdot\bigl(\A(y) (i\chi)\bigr)
    \qquad \text{or, equivalently,} \qquad
    -\nabla_y \cdot\bigl(\A(y) \nabla_y \varphi^{(1)}\bigr) = \nabla_y \cdot \A(y).
\end{align}
This is the cell-problem for the classical first-order corrector. That is, $\varphi^{(1)} = \mathbf{N}^{(1)}$ (cf.~\eqref{eqn:bloch_connection_chi2}).

\textbf{Step 8.} Finally, we obtain the homogenised equation by taking the Bloch data and applying (inverse) Fourier transform back into physical space. This gives us the formal equation for the \textit{spectral ansatz} $w_\eps$. (Compare this with \eqref{eqn:criminal_formal_homog_eqn_v1} for the criminal ansatz.)
\begin{align}\label{eqn:spectral_formal_homog_eqn_v1}
    \partial_{tt} w_\eps - \nabla \cdot \left( \A^{\hom,1} + \sum_{n=2}^\infty \A^{\hom,n} \odot (\eps \nabla)^{n-1} \right)\nabla w_\eps = f,
    \qquad
    w_\eps = w_\eps(t,x).
\end{align}

\begin{theorem}\label{thm:spectralansatz_longtime}
    \cite[Theorem~1]{duerinckx_gloria_ruf_2024_spectral_ansatz} Let the spectral correctors $\{\varphi^{(n)}\}_{n\geq 0}$ and $\{\zeta^{(n,m)} \}_{n,m\geq 0}$, and homogenised tensors $\{ \A^{\hom,n} \}_{n\geq 1}$ be defined as above. Let $f(t,x)$ satisfy Assumption \ref{assumption}. For each $k\geq 1$, define the spectral ansatz $w_\eps^{k}$ at level $k$ as the unique solution to a ``suitably regularized" version of the equation
    \begin{align}\label{eqn:spectral_formal_homog_eqn_v2}
        \partial_{tt} w_\eps^k - \nabla \cdot \left( \A^{\hom,1} + \sum_{n=2}^k \A^{\hom,n} \odot (\eps \nabla)^{n-1} \right)\nabla w_\eps = f.
    \end{align}
    Then, define the \textit{spectral two-scale expansion} at level $k$ by the expression
    \begin{align*}
        S_\eps^k[w_\eps^k,f] := \sum_{n=0}^k \eps^n \varphi^{(n)}\biggl(\frac{\cdot}{\eps}\biggr) \odot \psi_k(\eps\nabla) \nabla^n {\color{blue}w_\eps^k}
        + \underbrace{\eps^3 \sum_{2m=0}^{k-3} (-1)^m \eps^{2m} \sum_{n=0}^{k-3-2m} \eps^n \zeta^{(n,m)}\biggl(\frac{\cdot}{\eps}\biggr) \odot \psi_k(\eps\nabla) (\nabla^{n+1} \partial_t^{2m}) {\color{blue}f}}_{\text{This part contains information from the spectral ``bulk" (Step 5).}},
    \end{align*}
    where $\psi_k(\xi) :=\bigl\| \sum_{n=0}^k \varphi^{(n)} \odot (i\xi)^{\otimes n}\bigr\|^{-2}$ are Fourier multipliers satisfying $|\psi_k(\xi)|\leq 1$.
    Then there is a constant $C = C(k)$ such that
    \begin{align}\label{eqn:spectralansatz_longtime}
        \bigl\| u_\eps(t) - S_\eps^k[w_\eps^k,f](t) \bigr\|_{L^2(\R^d)} + 
        \left\| \nabla_{t,x} \left[ u_\eps(t) - S_\eps^k[w_\eps^k,f](t) \right] \right\|_{L^2(\R^d)}
        \leq (\eps C)^k \langle t \rangle \| \langle \nabla_{t,x} \rangle^{Ck} f \|_{L^1([0,t];L^2(\R^d))}.
    \end{align}
\end{theorem}

Just like Theorem \ref{thm:criminal_longtime} of the criminal ansatz, the proof of Theorem \ref{thm:spectralansatz_longtime} is a tedious affair, relying on an explicit formula for the remainder and standard stability estimates. We note that the proof of \eqref{eqn:spectralansatz_longtime} is done purely in ``physical space", meaning that the authors simply passed to the ``frequency space" in order to extract the Bloch data $\{\varphi^{(n)}\}_{n\geq 0}$ and $\{\zeta^{(n,m)} \}_{n,m\geq 0}$, forgetting about the space $L^2(Y)$ right after. This is atypical to spectral approaches to homogenisation (e.g.~spectral germ approach), seeking $\chi$-dependent estimates in frequency space, and  controlling the $\chi-$dependence during the Fourier/Gelfand inversion process. The estimate \eqref{eqn:spectralansatz_longtime} is the analogue of \eqref{eqn:criminal_longtime}, which allows for long time and arbitrary accuracy.

\paragraph*{Summary.} The spectral ansatz $S_\eps^k[w_\eps^k,f]$ is an asymptotic expansion for $u_\eps$ that is constructed by going into the frequency space and extracting the Bloch data. The Bloch data here refers to the spectral correctors $\{\varphi^{(n)}\}_{n\geq 0}$ and $\{\zeta^{(n,m)} \}_{n,m\geq 0}$, and homogenised tensors $\{ \A^{\hom,n} \}_{n\geq 1}$, and they are obtained by seeking an expansion for the eigenvalue equation $\mathcal{A}_{\eps\chi} \varphi_1^{\eps\chi} = \lambda_-^{\eps\chi} \varphi_1^{\eps\chi}$ in two-stages, first in powers of $\eps$, then in $(i\chi)$. The expression $S_\eps^k[w_\eps^k,f]$ consists of two terms: The first term involves $\{\varphi^{(n)}\}_{n\geq 0}$ and $\{ \A^{\hom,n} \}_{n\geq 1}$ and serves as a valid approximation on arbitrary long times, but with a maximum accuracy of $\mathcal{O}(\eps)$ (cf.~\eqref{eqn:criminal_longtime_v0only}). By including the second term of $S_\eps^k[w_\eps^k,f]$, which contains information of the spectral ``bulk" (at small frequencies), one is able to approximate $u_\eps$ to arbitrary long times and high accuracy \textit{simultaneously}.

\paragraph*{Ill-posed problems.} There is a final point of discussion pertaining to fact that \eqref{eqn:spectral_formal_homog_eqn_v2} has to be ``suitably regularized" before it can be uniquely solved. We have encountered this issue in the criminal ansatz, where a normal-form transformation + filtering step was applied to the formal equation \eqref{eqn:criminal_formal_homog_eqn_v1}. While this is not the only way to perturb the formal homogenised equation into a well-posed one, the problem of ill-posed equations appears in \textit{all} proposed ansatz for the long-time wave homogenisation, at this time of writing.

In \cite[Sect~1.3]{duerinckx_gloria_ruf_2024_spectral_ansatz}, the authors included a nice overview of the the ``tricks" available to obtain a well-posed equation. It was even shown that \eqref{eqn:spectral_formal_homog_eqn_v1} can be regularized in \textit{any} of the following ways:
\begin{itemize}
    \item High-frequency filtering: Perform a (spatial) Fourier cut-off on $f$. Used in criminal ansatz \cite{allaire_lamacz_rauch_2022_crime_pays}. 
    \item Higher-order regularization: Add a small but high-order positive operator so that the spatial part of \eqref{eqn:spectral_formal_homog_eqn_v1} is now uniformly elliptic. Used in first version of spectral ansatz \cite{benoit_gloria_2019_ballistic_transport}.
    \item Boussinesq trick: This relies on the perturbing the equation obtained from Bloch-data. For instance, the fourth-order (stationary) homogenised equation from Theorem \ref{thm:connection_evalue} is
    \begin{align}
        -\Div(\A^{\hom} \nabla v_\eps) + \eps^2 \mathbb{D} \odot \nabla^4 v_\eps = f,
    \end{align}
    which we know from Proposition \ref{prop:brunett_sign} is generally not well-posed. We shall replace the Brunett tensor $\mathbb{D}$ by $\mathcal{D} \in \R^{d\times d \times d \times d}$, where we pick a number $m\leq 0$ so that 
    \begin{align}
        \mathcal{D} \odot \xi^{\otimes 4} = \mathbb{D} \odot \xi^{\otimes 4} - (\A^{\hom} \xi \cdot \xi)(m \xi \cdot \xi) \geq 0, \qquad \text{for all $\xi \in \R^d$.}
    \end{align}
    The new (well-posed) homogenised equation is then
    \begin{align}
        -\Div(\A^{\hom} \nabla v_\eps) + \eps^2 \mathcal{D} \odot \nabla^4 v_\eps = f - \eps^2 m\Delta f.
    \end{align}
    This differs from higher-order regularization in that $f$ has been modified. Used in \cite{lamacz_1D, dohnal_lamacz_schweizer_2014, abdulle_pouchon_boussinesq_trick}.
\end{itemize}

\subsection{A summary table} 

For the reader's convenience, we summarize in Table \ref{tab:lit_summary} the key literature discussed in Section \ref{sect:improving_basic_homo}.

\renewcommand{\arraystretch}{1.3}
\begin{table}[h]
    \centering
    \captionsetup{width=1.0\linewidth} 
    \begin{tabular}{|c|c|c|c|} \hline 
        Year & Author(s) & Reference & Comments
        \\ \hline
        2002 & Conca, Orive, Vanninathan & \cite{conca_orive_vanninathan2002} & Connect Bloch and two-scale data to $\mathcal{O}(|\chi|^4)$.
        \\[2pt] \hline
        2009 & Birman, Suslina & \cite{birman_suslina_2009_hyperbolic} & \parbox{23em}{\vspace{0.1cm}Spectral germ. First norm-resolvent estimates for WE. \\ First hyperbolic results under this approach. $\cos(A^{1/2}t)$.\vspace{0.1cm}}
        \\[2pt] \hline
        2011 & Lamacz & \cite{lamacz_1D} & ~~\parbox{23em}{\vspace{0.1cm}Bloch expansion. One-dimensional setting. First rigorous proof of an ansatz that is good to $\mathcal{O}(\eps^{-2})$ in time.\vspace{0.1cm}}~~
        \\[2pt] \hline
        2019 & Benoit, Gloria & \cite{benoit_gloria_2019_ballistic_transport} & \parbox{23em}{\vspace{0.1cm}Spectral/Bloch ansatz. Long time, $\mathcal{O}(\eps)$ accuracy. \\ Applicable to stochastic setting.\vspace{0.1cm}}
        \\[2pt] \hline
        2021 & Meshkova & \cite{meshkova_2021_hyperbolic} & Spectral germ. Improvement on $A^{-1/2}\sin{(A^{1/2}t)}$.
        \\[2pt] \hline
        2022 & Allaire, Lamacz, Rauch & \cite{allaire_lamacz_rauch_2022_crime_pays} & \parbox{23em}{\vspace{0.1cm}Criminal ansatz. Long time, high accuracy. \\ Rigorous proof of the critical timescale $t\sim \eps^{-2}$ for the classical two-scale expansion. \vspace{0.1cm}}
        \\[2pt] \hline
        2023 & Duerinckx, Gloria, Ruf & \cite{duerinckx_gloria_ruf_2024_spectral_ansatz} & Spectral ansatz. Long time, high accuracy.
        \\[2pt] \hline 
    \end{tabular} 
    \caption{Various methods and their refinements.} \label{tab:lit_summary}
\end{table}

\section{Prototype one-dimensional problem and operator-norm resolvent estimates}
\label{prototype}

Here we return to the example discussed at the end of Introduction. We first represent the operator $A$ as the direct integral of a family of operators $A_\chi$ on the ``unit cell" $Y=[0,1],$ parametrised by the ``quasimomentum" $\chi\in Y'=[-\pi,\pi).$ These operators have compact resolvents and so their spectra are discrete (i.e. are sequences of finite-multiplicity eigenvalues accumulating at $\infty$). We then outline the Ryzhov triple framework \cite{ryzhov2020,  Physics, GrandePreuve, ChEK_future, kirill_survey, CKVZ_CMP}, which allows us to express each of these resolvents in terms of the Dirichlet-to-Neumann map at the ``vertices" (the pair of points at which the coefficient $a$ is discontinuous) and the resolvents of the Dirichlet operators on the two intervals where $a$ takes constant values. This re-frames the problem of homogenisation 
of the differential operator $A_\varepsilon$ on $L^2({\mathbb R})$ given by the differential expression 
\begin{equation}
-\frac{d}{dx}\Biggl\{a\biggl(\frac{x}{\varepsilon}\biggr)\frac{d}{dx}\Biggr\}
\label{Aeps_expr}
\end{equation}
as the question about the asymptotics of the lowest eigenvalue of a $\chi$-dependent $2\times2$ matrix and prove the related operator-norm convergence estimates.

The family $A_\chi$ representing the operators $A$  is obtained by invoking Gelfand transform (known also as Floquet-Bloch transform \cite{berkolaiko_kuchment_book}), which we recall next.

\subsection{Gelfand transform}\label{sect:boundary_triple_method}

In the context of differential operators with periodic coefficients, the following unitary transformation (``Gelfand transform", see \cite{Gelfand}) between $L^2({\mathbb R}^d)$ and $L^2(Y\times Y')$ has proved useful. For $u\in L^2({\mathbb R}^d)$ and every $\chi\in Y'$ that vanishes outside some ball, consider the periodic function
\[
\hat{u}(y,\chi):=\frac{1}{(2\pi)^{d/2}}\sum_{n\in{\mathbb Z}^d}u(y+n)\exp\bigl(-{\rm i}\chi\cdot(y+n)\bigr).
\] 
The inverse mapping is provided by the formula 
\begin{equation}
u(y)=\frac{1}{(2\pi)^{d/2}}\int_{Y'}\hat{u}(y,\chi)\exp({\rm i}\chi\cdot y)d\chi.
\label{inverse_gelfand}
\end{equation}
The operator $A$ is shown to be the direct integral of the operators $A_\chi$ defined by the differential expressions
\begin{equation}
-\biggl(\frac{d}{dy}+{\rm i}\chi\biggr)a(y)\biggl(\frac{d}{dy}+{\rm i}\chi\biggr),
\label{diff_expr_theta}
\end{equation} 
with domains 
\begin{equation*}
	\begin{aligned}
{\rm dom}(A_\chi)&=\Biggl\{u=u_-\oplus u_+\in H^2(0,l)\oplus H^2(l, 1):\\[0.4em] &u_-(0)=u_+(1), u_-(l)=u_+(l),\\[0.4em]
&a_-\biggl(\dfrac{d}{dy}+{\rm i}\chi\biggr)u_-\bigg\vert_{y=0}=a_+\biggl(\dfrac{d}{dy}+{\rm i}\chi\biggr)u_+\bigg\vert_{y=1},\quad  a_-\biggl(\dfrac{d}{dy}+{\rm i}\chi\biggr)u_-\bigg\vert_{y=l}=a_+\biggl(\dfrac{d}{dy}+{\rm i}\chi\biggr)u_+\bigg\vert_{y=l}
\Biggr\}.
\end{aligned}
\end{equation*}

Denote also by $\tilde{A}_\chi$ the operator given by the differential expression (\ref{diff_expr_theta}) with domain
\[
{\rm dom}(\tilde{A}_\chi)=\bigl\{u=u_-\oplus u_+\in H^2(0,l)\oplus H^2(l,1): u_-(0)=u_+(1), u_-(l)=u_+(l)\bigr\}.
\]

\subsection{Ryzhov triples and Krein's formula}

In the context of homogenisation (i.e. as $\varepsilon\to0$ above), operator-norm estimates for the Cauchy problem (\ref{Cauchy}) were obtained in \cite{birman_suslina_2009_hyperbolic, dorodnyi_suslina_2018_hyperbolic, meshkova_2021_hyperbolic} on the basis of analysing the ``spectral germ" of the related operator family $A=A_\varepsilon$ combined with the formula (\ref{hyperbolic_formula}). 


The operator $A_\chi$ is a (self-adjoint) restriction of the ``maximal" operator $\tilde{A}_\chi.$ Denote by $\Gamma_0^{(\chi)}, \Gamma_1^{(\chi)}: {\rm dom}(\tilde{A}_\chi)\to{\mathbb C}^2$ the Dirichlet and Neumann trace mappings:
\[
\Gamma_0^{(\chi)}:u\mapsto\left(\begin{array}{c}u_-(0)\\u_+(l)\end{array}\right),\qquad
\Gamma_1^{(\chi)}: u\mapsto\left(\begin{array}{c}a_+\biggl(\dfrac{d}{dy}+{\rm i}\chi\biggr)u_+\bigg\vert_{y=1}-a_-\biggl(\dfrac{d}{dy}+{\rm i}\chi\biggr)u_-\bigg\vert_{y=0}
\\[0.9em]a_+\biggl(\dfrac{d}{dy}+{\rm i}\chi\biggr)u_+\bigg\vert_{y=l}-a_-\biggl(\dfrac{d}{dy}+{\rm i}\chi\biggr)u_-\bigg\vert_{y=l}\end{array}\right).
\]
The domain of the (``minimal") operator $\tilde{A}_\chi^*$ then consists of $u\in{\rm dom}(\tilde{A}_\chi)$ such that $\Gamma_0^{(\chi)}u=\Gamma_1^{(\chi)}u=0.$

Consider the ``Dirichlet decoupling" operator $A_\chi^{(0)}$ given by the differential expression (\ref{diff_expr_theta}) on the domain
\[
{\rm dom}(A_\chi^{(0)})=\bigl\{u\in{\rm dom}(\tilde{A}_\chi): \Gamma_0^{(\chi)}u=0\bigr\}.
\]
In what follows, for an operator $A$ on $L^2(0,1),$ we denote by $\rho(A)$ the resolvent set of $A.$ 
For $z\in\rho(A_\chi^{(0)}),$ the Dirichlet-to-Neumann map (``$M$-matrix") $M_\chi(z),$  for the expressions (\ref{diff_expr_theta})
is defined as mapping the vector $\Gamma_0^{(\chi)}u$ of values at the ``vertices'' $0$ and $l$ to the vector of total fluxes (the sum of appropriately signed derivatives) $\Gamma_1^{(\chi)}$ at $0, l$ of the solution $u\in{\rm dom}(\tilde{A}_\chi)$ to the equation $\tilde{A}_\chi=zu.$ A direct  calculation yields
\[
M_\chi(z)=\left(\begin{array}{cc}
-k\sqrt{a_-}\cot\dfrac{kl}{\sqrt{a_-}}-k\sqrt{a_2}\cot\dfrac{k(1-l)}{\sqrt{a_2}}\  & \ \dfrac{{\rm e}^{{\rm i}\chi l}k\sqrt{a_-}}{\sin\dfrac{kl}{\sqrt{a_-}}}+\dfrac{{\rm e}^{-{\rm i}\chi(1-l)}k\sqrt{a_2}}{\sin\dfrac{k(1-l)}{\sqrt{a_-}}}\\[1.5em]
\dfrac{{\rm e}^{-{\rm i}\chi l}k\sqrt{a_-}}{\sin\dfrac{kl}{\sqrt{a_-}}}+\dfrac{{\rm e}^{{\rm i}\chi(1-l)}k\sqrt{a_2}}{\sin\dfrac{k(1-l)}{\sqrt{a_-}}}\ & \ -k\sqrt{a_-}\cot\dfrac{kl}{\sqrt{a_-}}-k\sqrt{a_2}\cot\dfrac{k(1-l)}{\sqrt{a_2}}
\end{array}\right).
\]
One has, for all $N=0,1,2,\dots,$
\begin{equation}
\begin{aligned}
M_\chi(z)&=\Lambda_\chi+z\Pi_\chi^*\bigl(I-z(A_\chi^{(0)})^{-1}\bigr)^{-1}\Pi_\chi=\Lambda_\chi+z\Pi_\chi^*\Pi_\chi+z^2\Pi_\chi^*(A_\chi^{(0)})^{-1}\bigl(I-z(A_\chi^{(0)})^{-1}\bigr)^{-1}\Pi_\chi\\[0.3em]
&=\Lambda_\chi+\sum_{j=0}^Nz^{j+1}\Pi_\chi^*\bigl(A_\chi^{(0)})^{-j}\Pi_\chi+z^{N+2}\Pi_\chi^*(A_\chi^{(0)})^{-N-1}\bigl(I-z(A_\chi^{(0)})^{-1}\bigr)^{-1}\Pi_\chi,
\end{aligned}
\label{M_expansion}
\end{equation}
where $\Lambda_\chi:=M_\chi(0)$ 
and $\Pi_\chi:{\mathbb C}^2\to{\rm dom}(\tilde{A}_\chi)$ is the ``lift" operator mapping vectors $\phi\in{\mathbb C}^2$ 
to the solution $u$ of the boundary value problem $\tilde{A}_\chi u=0,$ $\Gamma_0^{(\chi)}u=\phi.$ The ``boundary space" ${\mathbb C}^2$ and the ``boundary operators" $\Gamma_0^{(\chi)},$ $\Gamma_1^{(\chi)}$ constitute the ``classical" boundary triple \cite{Kochubei} for the operator $A_\chi.$ The triple $(A_\chi^{(0)}, \Lambda, \Pi),$ which we referred to as the ``Ryzhov triple" \cite{Ryzhov}, affords an extension of the approach   we discuss here to PDE settings. This is based on the formula (\ref{M_expansion}) and the celebrated ``Krein formula", which we recall next. For $\alpha, \beta\in{\mathbb C}^{2\times2},$ consider the operator $(A_\chi)_{\alpha,\beta}$ given by the differential expression (\ref{diff_expr_theta}) on the domain
\[
{\rm dom}(A_\chi)_{\alpha,\beta}=\bigl\{u\in{\rm dom}(\tilde{A}_\chi): \bigl(\alpha\Gamma_0^{(\chi)}+\beta\Gamma_1^{(\chi)}\bigr)u=0\bigr\}.
\]
(Note, in particular, that $A_\chi^{(0)}=(A_\chi)_{I,0}.$) For $z\in\rho\bigl((A_\chi)_{\alpha, \beta}\bigr)$ define the ``solution operator" $S_\chi(z)$ as the mapping $\phi\in{\mathbb C}^2$ as the solution to the boundary value problem $\tilde{A}_\chi u=zu,$ $\Gamma_0^{(\chi)}u=\phi.$ 
It is not difficult to see \cite{Ryzhov} that
\begin{equation}
S_\chi(z)=\bigl(I-z\bigl({A}_\chi^{(0)}\bigr)^{-1}\bigr)^{-1}\Pi_\chi,\quad z\in\rho\bigl({A}_\chi^{(0)}\bigr).
\label{Sform}
\end{equation}
Furthermore, the following identity (``Krein's formula") 
linking the resolvents of $(A_\chi)_{\alpha,\beta},$ $A_\chi^{(0)}$ and the $M$-matrix $M_\chi(z)$ holds: 
\begin{equation}
\bigl((A_\chi)_{\alpha, \beta}-zI\bigr)^{-1}=\bigl(A_\chi^{(0)}-zI\bigr)^{-1}-S_\chi(z)\bigl(\alpha+\beta M_\chi(z)\bigr)^{-1}S_\chi(\overline{z})^*,\qquad z\in\rho\bigl((A_\chi)_{\alpha, \beta}\bigr)\cap\rho\bigl(A_\chi^{(0)}\bigr).
\label{Krein_formula}
\end{equation}
We will use the formula (\ref{Krein_formula}) to study the asymptotics behaviour of the resolvents $(\varepsilon^{-2}(A_\chi)_{0,I}-zI)^{-1}$ as $\varepsilon\to0,$ aiming at approximation error estimates that are uniform with respect to $\chi\in[-\pi,\pi).$

\subsection{Operator-norm estimates in homogenisation via Krein's formula}

The matrix $\Lambda_\chi$ in (\ref{M_expansion}) is given by 
\[
\Lambda_\chi=\left(\begin{array}{cc}-D & \overline{\xi^{(\chi)}}\\[0.2em]\xi^{(\chi)} & -D\end{array}\right),
\]
where 
\[
\xi^{(\chi)}:=\frac{a_-}{l}{\rm e}^{-{\rm i}\chi l}+\frac{a_2}{1-l}{\rm e}^{{\rm i}\chi(1-l)},\qquad D:=\frac{a_-}{l}+\frac{a_2}{1-l}.
\]
The eigenvalues of $\Lambda_\chi$ are $\mu_{\Vert}^{(\chi)}=-D+|\xi^{(\chi)}|$ and $\mu_\perp^{(\chi)}=-D-|\xi^{(\chi)}|$ with the corresponding eigenfunctions given by 
\[
\psi^{(\chi)}_\Vert=\frac{1}{\sqrt{2}}\Biggl(1,\frac{\xi^{(\chi)}}{\vert\xi^{(\chi)}\vert}\Biggr)^\top, \qquad \psi^{(\chi)}_\perp=\frac{1}{\sqrt{2}}\Biggl(1,-\frac{\xi^{(\chi)}}{\vert\xi^{(\chi)}\vert}\Biggr)^\top.
\]
We denote by $\hat{\mathcal E}_\chi$ and $P_\chi$ the (one-dimensional) subspace of ${\mathbb C}^2$ generated by the vector $\psi_\Vert^{(\chi)}$ and the orthogonal projection from ${\mathbb C}^2$ onto this subspace, respectively.




For each $\chi\in Y',$ consider the ``truncated" lift operator 
$\hat{\Pi}_\chi:=\Pi_\chi P_\chi$ and the $\chi$-fibre $A_\chi^{\rm hom}:=-(\hat{\Pi}_\chi^*)^{-1}\Lambda_\chi\hat{\Pi}_\chi^{-1}$ of the homogenised operator. We also denote by $\Theta_\chi$ the orthogonal projection in $L^2(0,1)$ onto the range of $\hat{\Pi}_\chi.$ The following theorem, containing analogues of
\cite[Theorem]{simplified_method} and \cite[Theorem 5.2, Theorem 5.6]{CKVZ_CMP}, holds.
\begin{theorem}
\label{approximation_thm}
For every $\alpha\in(0,2),$ there exist $c, C_1, C_2>0$ such that:
\begin{enumerate}
\item
The (uniform in $\chi$) estimate 
\begin{equation}
{\rm dist}\Bigl(\sigma\bigl((A_\chi)_{0,I}\bigr), \sigma\bigl(A_\chi^{\rm hom}\bigr)\Bigr)\le C_1\chi^4.
\label{dist_est}
\end{equation}
for the distance between the spectra of $(A_\chi)_{0,I}$ and $A_\chi^{\rm hom}$ holds.
\item
For all $\chi\in Y'$ and $z\in{\mathbb C}$ such that ${\rm dist}\Bigl(z, \sigma\bigl(\varepsilon^{-2}(A_\chi)_{0,I}\bigr)\cup\sigma\bigl(\varepsilon^{-2}A_\chi^{\rm hom}\bigr)\Bigr)\ge1,$
$\vert z\vert\le c\varepsilon^{(\alpha-2)/2},$ one has
\begin{equation}
\Bigl\Vert
\bigl(\varepsilon^{-2}(A_{\chi})_{0,I}-zI\bigr)^{-1}-\bigl(\varepsilon^{-2}A_\chi^{\rm hom}-zI\bigr)^{-1}\Theta_\chi
\Bigr\Vert_{L^2(0,1)\to L^2(0,1)}\le C_2\varepsilon^\alpha,
\label{norm_est}
\end{equation}
where the approximating operator is understood as vanishing on the orthogonal complement of the range of $\hat{\Pi}_\chi.$
\end{enumerate}
\end{theorem}
\begin{proof}
1. The asymptotics of the lowest eigenvalue of $(A_\chi)_{0,I}$ is established by following the argument of the proof of \cite[Lemma 6.2]{Physics}. That provides an $O(\chi^4)$ error estimate for the difference between the said eigenvalue and the (quadratic in $\chi$) leading-order term of $A_\chi^{\rm hom},$ see (\ref{Ahom_asymp}) below. By virtue of the asymptotics (\ref{Ahom_asymp}), the bound (\ref{dist_est}) follows.    

2. Using the representation (\ref{M_expansion}), we write
\begin{align*}
\varepsilon^{-2}M_\chi(\varepsilon^2z)
&=\varepsilon^{-2}\Lambda_\chi+z\Pi_\chi^*\bigl(I-z\varepsilon^2(A_\chi^{(0)})^{-1}\bigr)^{-1}\Pi_\chi\\[0.3em]
&=\varepsilon^{-2}P_\chi\Lambda_\chi P_\chi+zP_\chi\Pi_\chi^*\Pi_\chi P_\chi+z^2\varepsilon^2P_\chi\Pi_\chi^*(A_\chi^{(0)})^{-1}\bigl(I-z\varepsilon^2(A_\chi^{(0)})^{-1}\bigr)^{-1}\Pi_\chi P_\chi\\[0.3em]
&\hspace{2cm}+\varepsilon^{-2}(1-P_\chi)M_\chi(\varepsilon^2z)(1-P_\chi),
\end{align*}
and therefore 
\begin{equation}
\begin{aligned}
\varepsilon^2M_\chi(\varepsilon^2z)^{-1}&=\varepsilon^2P_\chi M_\chi(\varepsilon^2z)^{-1}P_\chi+
\varepsilon^2(I-P_\chi) M_\chi(\varepsilon^2z)^{-1}(I-P_\chi)\\[0.3em]
&=\bigl(\varepsilon^{-2}P_\chi\Lambda_\chi P_\chi+zP_\chi\Pi_\chi^*\Pi_\chi P_\chi\bigr)^{-1}+O\bigl((|z|^2+1)\varepsilon^2\bigr).
\end{aligned}
\label{epsM_expansion}
\end{equation}

 By virtue of the the representation (\ref{Sform}) and Krein's formula (\ref{Krein_formula}), where we set $\alpha=0,$ $\beta=I,$ we now have
\begin{align}
\bigl(\varepsilon^{-2}(A_{\chi})_{0,I}-zI\bigr)^{-1}&=\bigl(\varepsilon^{-2}A_{\chi}^{(0)}-zI\bigr)^{-1}-\varepsilon^2S_\chi(\varepsilon^2z)M_\chi(\varepsilon^2z)^{-1}S_\chi(\varepsilon^2\overline{z})^*\nonumber\\[0.3em]
&=O(\varepsilon^2)-\bigl(\Pi_\chi+O(\varepsilon^2z)\bigr)\Bigl\{\bigl(\varepsilon^{-2}P_\chi\Lambda_\chi P_\chi+zP_\chi\Pi_\chi^*\Pi_\chi P_\chi\bigr)^{-1}+O\bigl((|z|^2+1)\varepsilon^2\bigr)\Bigr\}\bigl(\Pi_\chi^*+O(\varepsilon^2|z|)\bigr)\nonumber\\[0.3em]
&=\hat{\Pi}_\chi(\hat{\Pi}_\chi)^{-1}\Bigl\{-\varepsilon^{-2}(\hat{\Pi}_\chi^*)^{-1}\Lambda_\chi\hat{\Pi}_\chi^{-1}-zI\Bigr\}^{-1}(\hat{\Pi}_\chi^*)^{-1}\hat{\Pi}_\chi^*+O\bigl((|z|^2+|z|+1)\varepsilon^2\bigr)\nonumber\\[0.3em]
&=\bigl(\varepsilon^{-2}A_\chi^{\rm hom}-zI\bigr)^{-1}\bigr\vert_{\Pi_\chi\hat{\mathcal E}_\chi}+O\bigl((|z|^2+1)\varepsilon^2\bigr),\label{order1_app}
\end{align}
where the first term is extended to $L^2(0,1)$ by linearity so that the  extension vanishes on the ortogonal complement of $\Pi_\chi\hat{\mathcal E}_\chi.$
Hence, one has
\[
\bigl(\varepsilon^{-2}(A_{\chi})_{0,I}-zI\bigr)^{-1}=\bigl(A_\chi^{\rm hom}-zI\bigr)^{-1}\bigr\vert_{\Pi_\chi\hat{\mathcal E}_\chi}+O\bigl(\varepsilon^\alpha)
\]
as long as $|z|\le c\varepsilon^{(\alpha-2)/2}.$ This is equivalent to (\ref{norm_est}) by virtue of the definition of $\Theta_\chi.$
\end{proof}

\begin{proposition}
\label{Ahomprop}
The operator $A^{\rm hom}_\chi$
is the multiplication by
\begin{equation}
\begin{aligned}
&\dfrac{6\Bigl(D-\bigl\vert\xi^{(\chi)}\bigr\vert\Bigr)
}
{2+\bigl\vert\xi^{(\chi)}\bigr\vert^{-1}\biggl(a_-+a_2+\biggl(\dfrac{1-l}{l}a_-+\dfrac{l}{1-l}a_2\biggr)\cos\chi\biggr)}\\[0.4em]
&\hspace{2.5cm}=\biggl(\dfrac{l}{a_-}+\dfrac{1-l}{a_2}\biggr)^{-1}\biggl(\chi^2+\dfrac{a_-a_2(1-l)l+(1-2l)\bigl(a_-^2(1-l)^2-a_2l^2\bigr)}{12\bigl(a_-(1-l)+a_2l\bigr)^2}\chi^4\biggr)+O(\chi^6).
\end{aligned}
\label{Ahom_asymp}
\end{equation}
\end{proposition}
\begin{proof}
Consider the projection of $\Gamma_1^{(\chi)}\Pi_\chi\psi^{(\chi)}_\Vert$ onto the one-dimensional subspace of ${\mathbb C}^2$ generated by $\psi^{(\chi)}_\Vert:$
\[
P_\chi\Gamma_1^{(\chi)}\Pi_\chi\psi^{(\chi)}_\Vert=\Bigl\langle\Gamma_1^{(\chi)}\Pi_\chi\psi^{(\chi)}_\Vert,\psi^{(\chi)}_\Vert\Bigr\rangle\psi^{(\chi)}_\Vert.
\]
We are interested in the element of $\Pi_\chi\hat{\mathcal E}_\chi,$ i.e., a function of the form $\zeta\Pi_\chi\psi^{(\chi)}_\Vert$ such that
\[
P_\chi\Gamma_1^{(\chi)}\bigl(A_\chi^{(0)}\bigr)^{-1}\Bigl(\zeta\Pi_\chi\psi^{(\chi)}_\Vert\Bigr)=\Bigl\langle\Gamma_1^{(\chi)}\Pi_\chi\psi^{(\chi)}_\Vert,\psi^{(\chi)}_\Vert\Bigr\rangle\psi^{(\chi)}_\Vert.
\]
Taking the inner product of both sides of the last expression with $\psi^{(\chi)}_\Vert$ yields
\begin{equation}
\zeta=\dfrac{\Bigl\langle\Gamma_1^{(\chi)}\Pi_\chi\psi^{(\chi)}_\Vert,\psi^{(\chi)}_\Vert\Bigr\rangle}{\Bigl\langle\Gamma_1^{(\chi)}\Bigl(A_\chi^{(0)}\bigr)^{-1}\Pi_\chi\psi^{(\chi)}_\Vert,\psi^{(\chi)}_\Vert\Bigr\rangle}
\label{zeta}
\end{equation}
The function $\Pi_\chi\psi^{(\chi)}_\Vert$ solves
\[
-\biggl(\dfrac{d}{dy}+{\rm i}\chi\biggr)^2u=0,\qquad u(0)=u(1)=\dfrac{1}{\sqrt{2}},\quad u(l)=-\dfrac{\xi^{(\chi)}}{\sqrt{2}\vert\chi^{(\chi)}\vert}.
\]
By a direct calculation, we obtain 
\begin{equation*}
\Pi_\chi\psi^{(\chi)}_\Vert(y)=\dfrac{{\rm e}^{-{\rm i}\chi y}}{\sqrt{2}}\left\{\begin{array}{ll}
-\dfrac{1}{l}\Biggl({\rm e}^{{\rm i}\chi l}\dfrac{\xi^{(\chi)}}{\vert\xi^{(\chi)}\vert}+1\Biggr)y+1,\qquad y\in(0,l),\\[0.9em]
-\dfrac{1}{l-1}\Biggl\{\biggl({\rm e}^{{\rm i}\chi l}\dfrac{\xi^{(\chi)}}{\vert\xi^{(\chi)}\vert}+{\rm e}^{{\rm i}\chi}\biggr)y+{\rm e}^{{\rm i}\chi l}\dfrac{\xi^{(\chi)}}{\vert\xi^{(\chi)}\vert}+l{\rm e}^{{\rm i}\chi}\Biggr\},\qquad y\in(l,1).
\end{array}\right.
\end{equation*}
Denote by $u_-,$ $u_+$ the restrictions $u\vert_{(0,l)},$ $u\vert_{(l,1)},$ respectively. The Neumann trace operator is given by 
\begin{equation}
\Gamma_1^{(\chi)}u=\left(\begin{array}{c}a_-\biggl(\dfrac{d}{dy}+{\rm i}\chi\biggr)u_-(0)-a_2\biggl(\dfrac{d}{dy}+{\rm i}\chi\biggr)u_+(1)\\[0.8em]
a_2\biggl(\dfrac{d}{dy}+{\rm i}\chi\biggr)u_+(l)-a_-\biggl(\dfrac{d}{dy}+{\rm i}\chi\biggr)u_-(l)
\end{array}\right).
\label{Neumann_trace}
\end{equation}
Setting $u=\Pi_\chi\psi^{(\chi)}_\Vert$ in (\ref{Neumann_trace}), we obtain
\begin{align}
\Bigl\langle\Gamma_1^{(\chi)}\Pi_\chi\psi^{(\chi)}_\Vert,\psi^{(\chi)}_\Vert\Bigr\rangle&=-\dfrac{a_-}{l}
\Biggl(
1+\Re\biggl(\dfrac{\xi^{(\chi)}}{\vert\xi^{(\chi)}\vert}
{\rm e}^{{\rm i}\chi l}\biggr)
\Biggr)-
\dfrac{a_2}{1-l}
\Biggl(1+\Re\biggl(\dfrac{\xi^{(\chi)}}{\vert\xi^{(\chi)}\vert}{\rm e}^{{\rm i}\chi (l-1)}\biggr)\Biggr)=-D+\bigl\vert\xi^{(\chi)}\bigr\vert\label{mu_first_line}\\[0.4em]
&=\sqrt{\dfrac{a_-^2}{l^2}+\dfrac{a_2^2}{(1-l)^2}+\dfrac{2a_-a_2}{l(1-l)}\cos\chi}-\biggl(\dfrac{a_-}{l}+\dfrac{a_2}{1-l}\biggr).\nonumber
\end{align}

\begin{remark}
As, by definition, $\Lambda_\chi=M_\chi(0)=\Gamma_1^{(\chi)}\Pi_\chi,$ and $\psi^{(\chi)}_{\Vert}$ is an eigenvector of $\Lambda_\chi$ corresponding to the eigenvalue $\mu_{\vert}^{(\chi)},$ we have 
\[
\Bigl\langle\Gamma_1^{(\chi)}\Pi_\chi\psi^{(\chi)}_\Vert,\psi^{(\chi)}_\Vert\Bigr\rangle=\Bigl\langle\mu_{\Vert}^{(\chi)}\psi^{(\chi)}_\Vert,\psi^{(\chi)}_\Vert\Bigr\rangle=\mu_{\Vert}^{(\chi)}=\bigl\vert\xi^{(\chi)}\bigr\vert-D,
\]
which coincides with (\ref{mu_first_line}).
\end{remark}

Proceeding to the denominator of (\ref{zeta}), note first that the function $W:=\bigl(A_\chi^{(0)}\bigr)^{-1}\Pi_\chi\psi^{(\chi)}_\Vert=W_1\oplus W_2$ solves
\[
-a_{1,2}\biggl(\dfrac{d}{dy}+{\rm i}\chi\biggr)^2W_{1,2}=\Pi_\chi\psi^{(\chi)}_\Vert=:f,\qquad W_1(0)=W_1(l)=W_2(0)=W_2(1)=0.
\]
Consider the functions 
\[
h(y):=\int_0^yf(s)\int_s^ya^{-1},\ \  y\in(0,1),\qquad
g_1(y)=\left\{\begin{array}{ll}1-\dfrac{y}{l},\ \  &y\in(0,1),\\[0.6em]{\rm e}^{{\rm i}\chi}\dfrac{l-y}{l-1},\ \  &y\in(l,1),\end{array}\right.\quad g_2(y)=\left\{\begin{array}{ll}\dfrac{y}{l},\ \  &y\in(0,1),\\[0.6em]\dfrac{y-1}{l-1},\ \  &y\in(l,1).\end{array}\right.
\]

A direct calculation yields
\[
W(y)={\rm e}^{-{\rm i}\chi y}\Bigl\{
\bigl({\rm e}^{{\rm i}\chi}-1\bigr)^{-1}h(1)\bigl(1-g_1(y)-g_2(y)\bigr)+h(y)-h(l)g_2(y)
\Bigr\},\quad y\in(0,1).
\]
Setting $u=W$ in (\ref{Neumann_trace}), we obtain
\[
\Gamma_1^{(\chi)}W=\left(\begin{array}{c}\Biggl(\dfrac{a_-}{l}+\dfrac{a_2}{1-l}{\rm e}^{-{\rm i}\chi}\Biggr)h(l)+{\rm e}^{-{\rm i}\chi}\Biggl(\dfrac{a_2}{1-l}h(1)-\int_0^1f\Biggr)\\[0.6em]
-{\rm e}^{-{\rm i}\chi l}\Biggl\{\biggl(\dfrac{a_-}{l}+\dfrac{a_2}{1-l}\biggr)h(l)+\dfrac{a_2}{1-l}h(1)\Biggr\}
\end{array}\right)
\]
Finally, the inner product of the latter vector with $\psi^{(\chi)}_\Vert$ is 
\begin{equation}
\begin{aligned}
\Bigl\langle\Gamma_1^{(\chi)}\Bigl(A_\chi^{(0)}\bigr)^{-1}\Pi_\chi\psi^{(\chi)}_\Vert,\psi^{(\chi)}_\Vert\Bigr\rangle&=\dfrac{1}{6}\Biggl\{2+\Re\Biggl(
\bigl(l{\rm e}^{{\rm i}\chi l}+(1-l){\rm e}^{{\rm i}\chi(l-1)}\bigr)
\dfrac{\xi^{(\chi)}}{\vert\xi^{(\chi)}\vert}\Biggr)\Biggr\}\\[0.4em]
&=\dfrac{1}{6}\Biggl\{2+\dfrac{1}{\bigl\vert\xi^{(\chi)}\bigr\vert}\biggl(a_-+a_2+\biggl(\dfrac{1-l}{l}a_-+\dfrac{l}{1-l}a_2\biggr)\cos\chi\biggr)\Biggr\}.
\end{aligned}
\label{denominator}
\end{equation}
Combining (\ref{zeta}), (\ref{mu_first_line}), and (\ref{denominator}) yields the value in the statement of the proposition.
\end{proof}

\section{Hyperbolic evolution for the prototype operator with rapidly oscillating coefficients}
\label{hyperbolic_sec}

Here we combine the estimates obtained in the preceding section with the representation (\ref{hyperbolic_formula}) to study the behaviour of solutions to the hyperbolic evolution problem for the operators $A_\varepsilon,$ see (\ref{Aeps_expr}). We focus on the case $u_{\rm init}=0,$ $f=0$, thus considering the operator $A_\varepsilon^{-1/2}\sin(A_\varepsilon^{1/2}t)$ that enters the second term in (\ref{hyperbolic_formula}).

\subsection{Convergence estimate for the Cauchy problem}
\label{first_order_estimates}

Consider $\alpha\in(0,2)$ as above and suppose first that $\vert\chi|\le C_3\varepsilon^{(\alpha+2)/4}$ for some ($\chi$-independent) $C_3>0,$ which we choose below. 
By the first part of Theorem \ref{approximation_thm}, there exists a circle 
of radius $2C_1C_3^4\varepsilon^{-2}\chi^4\le2C_1C_3^4\varepsilon^\alpha\le 2C_1C_3^4$ (where $C_1$ is provided by (\ref{dist_est})) whose interior contains $\varepsilon^{-2}A_\chi^{\rm hom}$ as well as the lowest eigenvalue of the operator $\varepsilon^{-2}(A_\chi)_{0,I}.$ In particular, there exists a circle $\gamma$ of radius $R:=2\max\{C_1C_3^4,1\}$   whose interior contains $\varepsilon^{-2}A_\chi^{\rm hom}$ as well as the lowest eigenvalue of the operator $\varepsilon^{-2}(A_\chi)_{0,I}$ and additionally
\[
{\rm dist}\Bigl(z, \sigma\bigl(\varepsilon^{-2}(A_\chi)_{0,I}\bigr)\cup\sigma\bigl(\varepsilon^{-2}A_\chi^{\rm hom}\bigr)\Bigr)\ge 1,\quad z\in\gamma.
\]
Denote by $P$ the projection onto the corresponding eigenvector of $(A_\chi)_{0,I}.$  Using the Dunford-Schwartz calculus \cite{dunford_schwartz3}, we have 
\begin{equation}
\begin{aligned}		&\varepsilon\bigl((A_\chi)_{0,I}\bigr)^{-1/2}\sin\Bigl(\varepsilon^{-1}\bigl((A_\chi)_{0,I}\bigr)^{1/2}t\Bigr)
=\varepsilon P\bigl((A_\chi)_{0,I}\bigr)^{-1/2}\sin\Bigl(\varepsilon^{-1}\bigl((A_\chi)_{0,I}\bigr)^{1/2}t\Bigr)P\\[0.3em]
&\hspace{5cm}+\varepsilon(I-P)\bigl((A_\chi)_{0,I}\bigr)^{-1/2}\sin\Bigl(\varepsilon^{-1}\bigl((A_\chi)_{0,I}\bigr)^{1/2}t\Bigr)(1-P)\\[0.3em]
&\hspace{0cm}=-\frac{1}{2\pi{\rm i}}\ointctrclockwise_{\gamma}\dfrac{\sin\bigl(\sqrt{z}t\bigr)}{\sqrt{z}}\Bigl(\varepsilon^{-2}(A_\chi)_{0,I}-zI\Bigr)^{-1}dz
+\varepsilon(I-P)\bigl((A_\chi)_{0,I}\bigr)^{-1/2}\sin\Bigl(\varepsilon^{-1}\bigl((A_\chi)_{0,I}\bigr)^{1/2}t\Bigr)(1-P).
\end{aligned}
\label{decomposition}
\end{equation}

Next, note that by virtue of Proposition \ref{Ahomprop}, for $z\in\gamma$ one has $|z|\le C_4\varepsilon^{-2}\chi^2\le C_4C_3^2\varepsilon^{(\alpha-2)/2}$ for some $C_4>0.$ We choose $C_3=cC_4^{-1/2}$ so that $C_4C_3^2=c.$ Using the second part of Theorem 
\ref{approximation_thm} then yields
\begin{equation}
\begin{aligned}
&\Biggl\Vert\ointctrclockwise_{\gamma}\dfrac{\sin\bigl(\sqrt{z}t\bigr)}{\sqrt{z}}\biggl\{\Bigl(\varepsilon^{-2}(A_\chi)_{0,I}-zI\Bigr)^{-1}-\bigl(\varepsilon^{-2}A_\chi^{\rm hom}-zI\bigr)^{-1}\Theta_\chi\biggr\}dz\Biggr\Vert_{L^2(0,1)\to L^2(0,1)}
\\[0.3em]
&\hspace{4cm}
\le C_2\varepsilon^\alpha\ointctrclockwise_{\gamma}\biggl\vert\dfrac{\sin\bigl(\sqrt{z}t\bigr)}{\sqrt{z}}\biggr\vert dz\le 2\pi\tilde{R}C_2\varepsilon^\alpha\min\Biggl\{t,\sqrt{\frac{2}{C_4}}\frac{\varepsilon}{|\chi|}\Biggr\}
\end{aligned}
\label{contour}
\end{equation}
Furthermore, one clearly has 
\begin{equation}
\Bigl\Vert(I-P)\bigl((A_\chi)_{0,I}\bigr)^{-1/2}\sin\Bigl(\varepsilon^{-1}\bigl((A_\chi)_{0,I}\bigr)^{1/2}t\Bigr)(1-P)\Bigr\Vert_{L^2(0,1)\to L^2(0,1)}\le 1.
\label{remainder_operator}
\end{equation}
It follows from (\ref{decomposition}), (\ref{contour}), and (\ref{remainder_operator}) that
\begin{equation}
\begin{aligned}
&
\Bigl\Vert\varepsilon\bigl((A_\chi)_{0,I}\bigr)^{-1/2}\sin\Bigl(\varepsilon^{-1}\bigl((A_\chi)_{0,I}\bigr)^{1/2}t\Bigr)
-\varepsilon(A_\chi^{\rm hom})^{-1/2}\sin\Bigl(\varepsilon^{-1}(A_\chi^{\rm hom})^{1/2}t\Bigr)\Theta_\chi\Bigr\Vert_{L^2(0,1)\to 
L^2(0,1)}
\\[0.3em]
&\hspace{4cm}
\le\varepsilon+2\pi\tilde{R}C_2\varepsilon^\alpha\min\Biggl\{t,\sqrt{\frac{2}{C_4}}\frac{\varepsilon}{|\chi|}\Biggr\}
\end{aligned}
\label{small_chi}
\end{equation}

Finally, if $|\chi|>C_3\varepsilon^{(\alpha+2)/4}$ then for some $C_5>0$ one has
\begin{equation}
\max\biggl\{\Bigl\Vert \bigl((A_\chi)_{0,I}\bigr)^{-1/2}\Bigr\Vert_{L^2\to L^2}, \Bigl\Vert (A_\chi^{\rm hom})^{-1/2}\Bigr\Vert_{L^2\to L^2}\biggr\}\le C_5|\chi|^{-1}\le C_5C_3^{-1}\varepsilon^{-(\alpha+2)/4}.
\label{inverse_root_est}
\end{equation}
(Note that $C_5$ and $C_4$ can be replaced by a single constant at the expense of possibly increasing $C_4.$)
Combining (\ref{small_chi}) and (\ref{inverse_root_est}) yields 
\begin{equation*}
\Bigl\Vert\varepsilon\bigl((A_\chi)_{0,I}\bigr)^{-1/2}\sin\Bigl(\varepsilon^{-1}\bigl((A_\chi)_{0,I}\bigr)^{1/2}t\Bigr)
-\varepsilon(A_\chi^{\rm hom})^{-1/2}\sin\Bigl(\varepsilon^{-1}(A_\chi^{\rm hom})^{1/2}t\Bigr)\Theta_\chi\Bigr\Vert_{L^2(0,1)\to L^2(0,1)}
\le E^{(1)}(\varepsilon,\chi, t),
\end{equation*}
where 
\begin{equation}
E^{(1)}(\varepsilon,\chi, t):=\left\{\begin{array}{ll}\varepsilon+\tilde{R}C_2\varepsilon^\alpha\min\Biggl\{t,\sqrt{\dfrac{2}{C_4}}\dfrac{\varepsilon}{|\chi|}\Biggr\}
\qquad &{\rm if}\ \  |\chi|\le\varepsilon^{(\alpha+2)/4},\\[0.6em]
2C_5C_3^{-1}\varepsilon^{1-(\alpha+2)/4}\qquad &{\rm if}\ \   \varepsilon^{(\alpha+2)/4}\le|\chi|\le\pi.\end{array}\right. 
\label{E1}
\end{equation}
For $\alpha\in(1,2),$ 
the second-order approximation leads to a convergence estimate (as $\varepsilon\to0$) up to the times of order $\varepsilon^{-\alpha+\sigma},$ 
for all $\sigma>0.$ The corresponding error (uniform with respect to $\chi\in Y'$) is obtained from (\ref{E2}) as being of the order $O(\varepsilon^{\min\{1-(\alpha+2)/4, \sigma\}})=O(\varepsilon^{\min\{(2-\alpha)/4, \sigma\}}).$

\subsection{Second-order matrix approximation}\label{sect:boundary_triple_second_order}


We follow the approach of the proof of Theorem \ref{approximation_thm} and expand $\varepsilon^{-2}M_\chi(\varepsilon^2z)$ to the term of order $O(|z|^2\varepsilon^2).$ In particular, using the representation (\ref{M_expansion}), we write
\begin{align*}
\varepsilon^{-2}M_\chi(\varepsilon^2z)
&=\varepsilon^{-2}\Lambda_\chi+z\Pi_\chi^*\bigl(I-z\varepsilon^2(A_\chi^{(0)})^{-1}\bigr)^{-1}\Pi_\chi\\[0.3em]
&=\varepsilon^{-2}P_\chi\Lambda_\chi P_\chi+zP_\chi\Pi_\chi^*\Pi_\chi P_\chi+z^2\varepsilon^2P_\chi\Pi_\chi^*(A_\chi^{(0)})^{-1}\Pi_\chi P_\chi+z^3\varepsilon^4P_\chi\Pi_\chi^*(A_\chi^{(0)})^{-2}\bigl(I-z\varepsilon^2(A_\chi^{(0)})^{-1}\bigr)^{-1}\Pi_\chi P_\chi\\[0.3em]
&\hspace{2cm}
+\varepsilon^{-2}(1-P_\chi)M_\chi(\varepsilon^2)(1-P_\chi),
\end{align*}
and therefore 
\begin{equation}
\begin{aligned}
\varepsilon^2M_\chi(\varepsilon^2z)^{-1}&=\varepsilon^2P_\chi M_\chi(\varepsilon^2z)^{-1}P_\chi+
\varepsilon^2(I-P_\chi) M_\chi(\varepsilon^2z)^{-1}(I-P_\chi)\\[0.3em]
&=\bigl(\varepsilon^{-2}P_\chi\Lambda_\chi P_\chi+zP_\chi\Pi_\chi^*\Pi_\chi P_\chi+z^2\varepsilon^2P_\chi\Pi_\chi^*(A_\chi^{(0)})^{-1}\Pi_\chi P_\chi\bigr)^{-1}+O\bigl((|z|^3\varepsilon^2+1)\varepsilon^2\bigr).
\end{aligned}
\label{epsM_expansion1}
\end{equation}

Denote $\hat{A}_\chi^{(0)}:=\bigl(\bigl({A}_\chi^{(0)}\bigr)^{-1}\bigr\vert_{\Pi_\chi\hat{\mathcal E}_\chi}\bigr)^{-1}.$ By virtue of the the representation (\ref{Sform}) and Krein's formula (\ref{Krein_formula}), where we set $\alpha=0,$ $\beta=I,$ we now have
\begin{align}
\bigl(\varepsilon^{-2}(A_{\chi})_{0,I}-zI\bigr)^{-1}&=\bigl(\varepsilon^{-2}A_{\chi}^{(0)}-zI\bigr)^{-1}-\varepsilon^2S_\chi(\varepsilon^2z)M_\chi(\varepsilon^2z)^{-1}S_\chi(\varepsilon^2\overline{z})^*\nonumber\\[0.3em]
&=O(\varepsilon^2)-\bigl(\Pi_\chi+O(\varepsilon^2z)\bigr)\Bigl\{\bigl(\varepsilon^{-2}P_\chi\Lambda_\chi P_\chi+zP_\chi\Pi_\chi^*\Pi_\chi P_\chi+z^2\varepsilon^2P_\chi\Pi_\chi^*(A_\chi^{(0)})^{-1}\Pi_\chi P_\chi\bigr)^{-1}\nonumber\\[0.3em]
&\hspace{6cm}+O\bigl((|z|^3\varepsilon^2+1)\varepsilon^2\bigr)\Bigr\}\bigl(\Pi_\chi^*+O(\varepsilon^2|z|)\bigr)\nonumber\\[0.3em]
&=\hat{\Pi}_\chi(\hat{\Pi}_\chi)^{-1}\Bigl(
\varepsilon^{-2}A_\chi^{\rm hom}-z-z^2\varepsilon^2\bigl(\hat{A}_\chi^{(0)}\bigr)^{-1}\Bigr)^{-1}(\hat{\Pi}_\chi^*)^{-1}\hat{\Pi}_\chi^*+O\bigl((|z|^3\varepsilon^2+|z|+1)\varepsilon^2\bigr)\nonumber
\\[0.3em]
&=\Bigl(\varepsilon^{-2}A_\chi^{\rm hom}-z-z^2\varepsilon^2\bigl(\hat{A}_\chi^{(0)}\bigr)^{-1}\Bigr)^{-1}\bigr\vert_{\Pi_\chi\hat{\mathcal E}_\chi}+O\bigl((|z|^3\varepsilon^2+|z|+1)\varepsilon^2\bigr),\label{higher_order_exp}
\end{align}
where, similarly to (\ref{order1_app}), the first terms is extended by linearity to $L^2(0,1)$ so that the extension vanishes on the orthogonal complement of $\Pi_\chi\hat{\mathcal E}_\chi.$

We next determine a Jacobi matrix
\[
J=\left(\begin{array}{cc}q_0 &b_1\\
b_1&q_1\end{array}\right),\qquad q_0, q_1,  b_1\in{\mathbb R},
\]
and $c>0$ such that
\begin{equation}
cz-q_0-\dfrac{b_1^2}{cz-q_1}=-\varepsilon^{-2}A^{\rm hom}_\chi+z+z^2\varepsilon^2\bigl(\hat{A}_\chi^{(0)}\bigr)^{-1}+O\bigl(|z|^3\varepsilon^4\bigr),
\label{first_approximation}
\end{equation}
Noting that the resolvent equation 
\[
(J-cz)\left(\begin{array}{c}x_1\\x_2\end{array}\right)=\left(\begin{array}{c}f\\0\end{array}\right)
\]
is equivalent to
\[
-\biggl(cz-q_0-\dfrac{b_1^2}{cz-q_1}\biggr)x_1=f,
\]
we infer from (\ref{higher_order_exp}) that the operator 
$\bigl(\hat\Pi^*_\chi\bigr)^{-1}M(\varepsilon^2 z)^{-1}\hat{\Pi}_\chi$ (cf. (\ref{epsM_expansion}), (\ref{epsM_expansion1}))
is order $O(|z|^3\varepsilon^4)$ close to the generalised resolvent
\begin{equation}
{\mathcal R}_{\chi, \varepsilon}^{\rm app}:=\mathcal{I}_1^*\left\{\left(\begin{array}{cc}\varepsilon^{-2}A_\chi^{\rm hom} & 0\\[0.2em]0&0\end{array}\right)+\frac{1}{4}\varepsilon^{-2}\hat{A}_\chi^{(0)}\left(\begin{array}{cc}1 & \pm1\\[0.2em]\pm1 &1\end{array}\right)-\frac{z}{2}\right\}^{-1}\!\!\!\mathcal{I}_1
\label{gen_res}
\end{equation}
where 
the operator ${\mathcal I}_1$ maps $x_1\in{\mathbb C}$ to the vector $(x_1, 0)^\top\in{\mathbb C}^2,$ so that $\mathcal{I}_1:(x_1,x_2)\mapsto x_1.$

Indeed, expanding the left-hand side of (\ref{first_approximation}) in powers of $z$ and comparing the coefficients on either side in front of similar powers of $z$ yields a system of equations for the entries of the matrix $J:$
\begin{equation}
\dfrac{b_1^2}{q_1}-q_0=-\varepsilon^{-2}A^{\rm hom}_\chi,\qquad c\biggl(1+\dfrac{b_1^2}{q_1^2}\biggr)=1,\qquad c^2\dfrac{b_1^2}{q_1^3}=\varepsilon^2\bigl(\hat{A}_\chi^{(0)}\bigr)^{-1}.
\label{recurrence_system}
\end{equation}
The system (\ref{recurrence_system}) has infinitely many solutions $(c, q_0, q_1, b_1).$ One convenient option is to set 
\[
c=\dfrac{1}{2},\qquad q_0=\frac{1}{4}\varepsilon^{-2}\hat{A}_\chi^{(0)}+\varepsilon^{-2}A_\chi^{\rm hom},\qquad q_1=\frac{1}{4}\varepsilon^{-2}\hat{A}_\chi^{(0)},\qquad b_1^2=\dfrac{1}{16}\varepsilon^{-4}\bigl(\hat{A}_\chi^{(0)}\bigr)^2.
\]
The resolvent appearing between the projection operators in (\ref{gen_res}) is the resolvent of a self-adjoint operator on ${\mathbb C}^2$ (i.e. a symmetric matrix in the present setting). The latter can be viewed as a dilation of the space ${\mathbb C}$ in which the resolvent $(A_\chi^{\rm hom}-z)^{-1}$  of the first-order approximation acts. The sign choice in the off-diagonal entries in (\ref{gen_res}) corresponds to the transformation $(x_1, x_2)\mapsto (x_1, -x_2)$ of the dilation space ${\mathbb C}$ yielding a unitarily equivalent dilation operator. In what follows we choose the sign ``+" in (\ref{gen_res}). 

Denote 
\[
A^{{\rm hom},(2)}_\chi:=\left(\begin{array}{cc}A_\chi^{\rm hom} & 0\\[0.2em]0&0\end{array}\right)+\frac{1}{4}\hat{A}_\chi^{(0)}\left(\begin{array}{cc}1 & 1\\[0.2em]1 &1\end{array}\right),
\]
so that ${\mathcal R}_{\chi,\varepsilon}^{\rm app}={\mathcal I}_1^*\bigl(\varepsilon^{-2}A^{{\rm hom},(2)}_\chi-z/2\bigr)^{-1}{\mathcal I}_1,$ see (\ref{gen_res}) The eigenvalues of the matrix $J=\varepsilon^{-2}A^{{\rm hom},(2)}_\chi$
are given by
\begin{equation*}
\frac{z^{\pm}_\chi}{2}=\dfrac{1}{2\varepsilon^2}\biggl(A_\chi^{\rm hom}+\hat{A}_\chi^{(0)}/2\pm\sqrt{\bigl(A_\chi^{\rm hom}\bigr)^2+\bigl(\hat{A}_\chi^{(0)}\bigr)^2/4}\biggr).
\end{equation*}
For small $|\chi|,$ the two eigenvalues are estimated as follows:
\begin{equation*}
\frac{z_\chi^-}{2}=\frac{1}{2\varepsilon^2}\Bigl\{A_\chi^{\rm hom}-\bigl(\hat{A}_\chi^{(0)}\bigr)^{-1}\bigl(A_\chi^{\rm hom}\bigr)^2+O\Bigl(\bigl(A_\chi^{\rm hom}\bigr)^3\Bigr)\Bigr\},\qquad \frac{z_\chi^+}{2}=\frac{1}{2\varepsilon^2}\Bigl\{\hat{A}_\chi^{(0)}+A_\chi^{\rm hom}+O\Bigl(\bigl(A_\chi^{\rm hom}\bigr)^2\Bigr)\Bigr\}.
\end{equation*}
The eigenvalue $z_\chi^-/2$ behaves like $(2\varepsilon^2)^{-1}A_\chi^{\rm hom}$ for small $\varepsilon,\chi$ and, as we shall see, the contribution of the corresponding spectral projection to the asymptotics of the resolvent $(\varepsilon^{-2}(A_{\chi})_{0,I}-I)^{-1}$ corresponds to a quasimomentum range overlapping with that 
of the first-order approximation discussed in the preceding section. 
The contribution of the spectral projection for the eigenvalue $z_\chi^+/2$ to the asymptotics of the resolvent $(\varepsilon^{-2}(A_{\chi})_{0,I}-I)^{-1}$ is of order $O(\varepsilon^{2})$ and can therefore be included in the overall approximation error. 

Recalling the orthogonal projection $\Theta_\chi:L^2(0,1)\rightarrow {\rm Range}(\hat{\Pi}_\chi)$ in Theorem \ref{approximation_thm}, we have thus proved the following analogue of the second part of Theorem \ref{approximation_thm}.

\begin{theorem} 
\label{approximation_thm2}
For every $\alpha\in(0,4),$ there exist $\tilde{c}, \tilde{C}_2>0$ such that
for all $\chi\in Y'$ and $z\in{\mathbb C}$ satisfying ${\rm dist}\Bigl(z, \sigma\bigl(\varepsilon^{-2}(A_\chi)_{0,I}\bigr)\cup
\bigl\{z_\chi^-\bigr\}\Bigr)\ge\varepsilon^{\min\{0,2(\alpha-1)/3\}},$
$\vert z\vert\le \tilde{c}\varepsilon^{(\alpha-4)/3},$ one has
\begin{equation}
\begin{aligned}
\Bigl\Vert
\bigl(\varepsilon^{-2}(A_{\chi})_{0,I}-zI\bigr)^{-1}
-{\mathcal R}^{\rm app}_{\chi,\varepsilon}(z)\Theta_\chi
\Bigr\Vert_{L^2(0,1)\to L^2(0,1)}&= O\Bigl((|z|^3\varepsilon^4+|z|\varepsilon^2)\varepsilon^{\max\{0,4(1-\alpha)/3\}}+\varepsilon^2\Bigr)\\[0.3em]
&\le \tilde{C}_2\varepsilon^{\min\{(\alpha+2)/3, (4-\alpha)/3\}},
\end{aligned}
\label{norm_est2}
\end{equation}
where for each $z$, the operator ${\mathcal R}^{\rm app}_{\chi,\varepsilon}(z)$ is viewed as the multiplication by a constant on the range of $\hat{\Pi}_\chi.$
\end{theorem}

\subsection{Second-order error estimate for the Cauchy problem}

In what follows, $\alpha\in(0,4).$ Similarly to the approach of Section \ref{first_order_estimates}, suppose first that $\vert\chi|\le\tilde{C}_3\varepsilon^{(\alpha+2)/6}$ for some $\tilde{C}_3,$ which we choose in what follows, and 
note that 
$\varepsilon^{-2}\chi^4\le\varepsilon^{2(\alpha-1)/3}.$
 Therefore, there exists a circle $\tilde{\gamma}$ of radius $\tilde{R}\varepsilon^{\min\{0,2(\alpha-1)/3\}},$   
$\tilde{R}:=2\max\{C_1\tilde{C}_3^4,1\}$ (where $C_1$ is still provided by (\ref{dist_est})) whose interior contains $z_\chi^-$ as well as the lowest eigenvalue of $\varepsilon^{-2}(A_\chi)_{0,I},$ and in addition one has 
\[
{\rm dist}\Bigl(z, \sigma\bigl(\varepsilon^{-2}(A_\chi)_{0,I}\bigr)\cup\bigl\{z_\chi^-\bigr\}\bigr)\Bigr)\ge \varepsilon^{\min\{0,2(\alpha-1)/3\}},\quad z\in\tilde{\gamma}.
\]

There exists $\tilde{C}_4>0$ such that for $z\in\tilde{\gamma}$ one has $(\tilde{C}_4/2)\varepsilon^{-2}\chi^2\le|z|\le \tilde{C}_4\varepsilon^{-2}\chi^2\le \tilde{C}_4\tilde{C}_3^2\varepsilon^{(\alpha-4)/3}.$ Choosing $\tilde{C}_3=\tilde{c}(\tilde{C}_4)^{-1/2},$ we then have $|z|\le\tilde{c}\varepsilon^{(\alpha-4)/3}$ for all $z\in\tilde{\gamma}.$ Invoking Theorem \ref{approximation_thm2}, we obtain
\begin{equation}
\begin{aligned}
&\Biggl\Vert\ointctrclockwise_{\tilde{\gamma}}\dfrac{\sin\bigl(\sqrt{z}t\bigr)}{\sqrt{z}}\biggl\{\Bigl(\varepsilon^{-2}(A_\chi)_{0,I}-zI\Bigr)^{-1}-{\mathcal R}^{\rm app}_{\chi,\varepsilon}(z)\Theta_\chi\biggr\}dz\Biggr\Vert_{L^2(0,1)\to L^2(0,1)}\\[0.3em]
&\hspace{1cm}\le \tilde{C}_2\varepsilon^{\min\{(\alpha+2)/3, (4-\alpha)/3\}}\ointctrclockwise_{\tilde{\gamma}}\biggl\vert\dfrac{\sin\bigl(\sqrt{z}t\bigr)}{\sqrt{z}}\biggr\vert dz\le 2\pi\tilde{R}\tilde{C}_2\varepsilon^{(\alpha+2)/3}\min\Biggl\{t,\sqrt{\frac{2}{\tilde{C}_4}}\frac{\varepsilon}{|\chi|}\Biggr\}.
\end{aligned}
\label{contour2}
\end{equation}

It follows from (\ref{decomposition}), where $\gamma$ is replaced by $\tilde{\gamma}$, (\ref{remainder_operator}), and (\ref{contour2}) that
\begin{equation}
\begin{aligned}
&\Bigl\Vert\varepsilon\bigl((A_\chi)_{0,I}\bigr)^{-1/2}\sin\Bigl(\varepsilon^{-1}\bigl((A_\chi)_{0,I}\bigr)^{1/2}t\Bigr)
\\[0.3em]
&
\hspace{3.5cm}
-\sqrt{2}\varepsilon{\mathcal I}_1^*\bigl(A_\chi^{{\rm hom},(2)}\bigr)^{-1/2}\sin\Bigl(\varepsilon^{-1}\bigl(2A_\chi^{{\rm hom}, (2)}\bigr)^{1/2}t\Bigr){\mathcal I}_1\Theta_\chi\Bigr\Vert_{L^2(0,1)\to L^2(0,1)}
\\[0.2em]
&\hspace{7cm}
\le\varepsilon+\tilde{C}_2\varepsilon^{\min\{(\alpha+2)/3\}}\min\Biggl\{t,\sqrt{\frac{2}{\tilde{C}_4}}\frac{\varepsilon}{|\chi|}\Biggr\}.
\end{aligned}
\label{small_chi1}
\end{equation}

Furthermore, if $|\chi|>\tilde{C}_3\varepsilon^{(\alpha+2)/6}$ there exists $\tilde{C}_5>0$ such that
\[
\max\biggl\{\Bigl\Vert \bigl((A_\chi)_{0,I}\bigr)^{-1/2}\Bigr\Vert_{L^2(0,1)\to L^2(0,1)}, \sqrt{2}\Bigl\Vert \bigl(A_\chi^{{\rm hom}, (2)}\bigr)^{-1/2}\Bigr\Vert_{L^2(0,1)\to L^2(0,1)}\biggr\}\le \tilde{C}_5|\chi|^{-1}\le \tilde{C}_5\tilde{C}_3^{-1}\varepsilon^{-(\alpha+2)/6}.
\]
Combining this with (\ref{small_chi1}) yields  (cf. (\ref{inverse_root_est}))
\begin{equation}
\begin{aligned}
&
\Bigl\Vert\varepsilon\bigl((A_\chi)_{0,I}\bigr)^{-1/2}\sin\Bigl(\varepsilon^{-1}\bigl((A_\chi)_{0,I}\bigr)^{1/2}t\Bigr)
\\[0.3em]
&\hspace{1.5cm}
-\varepsilon\sqrt{2}{\mathcal I}_1^*\bigl(A_\chi^{{\rm hom}, (2)}\bigr)^{-1/2}\sin\Bigl(\varepsilon^{-1}\bigl(2A_\chi^{{\rm hom},(2)}\bigr)^{1/2}t\Bigr){\mathcal I}_1\Theta_\chi\Bigr\Vert_{L^2(0,1)\to L^2(0,1)}\le E^{(2)}(\varepsilon,\chi, t),
\end{aligned}
\label{evolest1}
\end{equation}
where (cf. (\ref{E1}))
\begin{equation}
E^{(2)}(\varepsilon,\chi, t):=\left\{\begin{array}{ll}\varepsilon+\tilde{C}_2\varepsilon^{(\alpha+2)/3}\min\Biggl\{t,\sqrt{\dfrac{2}{\tilde{C}_4}}\dfrac{\varepsilon}{|\chi|}\Biggr\}\qquad &{\rm if}\ \  |\chi|\le\varepsilon^{(\alpha+2)/6},\\[0.6em]
2\tilde{C}_5\tilde{C}_3^{-1}\varepsilon^{1-(\alpha+2)/6}\qquad &{\rm if}\ \   \varepsilon^{(\alpha+2)/6}\le|\chi|\le\pi.\end{array}\right. 
\label{E2}
\end{equation}


\subsection{Analysis of the second-order homogenised dynamics}
\label{second_order_dynamics}

Denote by $v^\pm_{(\chi)}$ 
normalised eigenvectors of the matrix $A_\chi^{{\rm hom},(2)}$ corresponding to the eigenvalues $z_\pm^{(\chi)}/2.$ Then one has
\[
\varepsilon^{-2}A_\chi^{{\rm hom}, (2)}=\bigl(v_\chi^-\  v_\chi^+\bigr)\left(\begin{array}{cc}z_\chi^-/2&0\\[0.2em]0&z_\chi^+/2\end{array}\right)\bigl(v_\chi^-\  v_\chi^+\bigr)^\top,
\]
where $(v_\chi^-\  v_\chi^+)$ is the matrix with columns $v_\chi^-,$ $v_\chi^+.$ It follows that 
\begin{align*}
\varepsilon\sqrt{2}{\mathcal I}_1^*\bigl(A_\chi^{{\rm hom}, (2)}\bigr)^{-1/2}\sin\Bigl(\varepsilon^{-1}\bigl(2A_\chi^{{\rm hom},(2)}\bigr)^{1/2}t\Bigr){\mathcal I}_1&=
2{\mathcal I}_1^*\bigl(v_\chi^-\  v_\chi^+\bigr)\left(\begin{array}{cc}\dfrac{\sin\bigl(\sqrt{z_\chi^-}t\bigr)}{\sqrt{z_\chi^-}}&0
\\[0.3em]0&\dfrac{\sin\bigl(\sqrt{z_\chi^+}t\bigr)}{\sqrt{z_\chi^+}}\end{array}\right)\bigl(v_\chi^-\  v_\chi^+\bigr)^\top{\mathcal I}_1\\[0.3em]
&=2\bigl((v_\chi^-)_1\bigr)^2\dfrac{\sin\bigl(\sqrt{z_\chi^-}t\bigr)}{\sqrt{z_\chi^-}}+O(\varepsilon),
\end{align*}
where 
$
(v_\chi^-)_1=\bigl(1+O\bigl(|\chi|^2\bigr)\bigr)/\sqrt{2}
$
is the first component of the vector $v_\chi^-.$
Using the fact that $z_\chi^-=\varepsilon^{-2}\bigl(A_\chi^{\rm hom}+O\bigl(|\chi|^4\bigr)\bigr),$ we obtain 
\begin{align*}
\varepsilon\sqrt{2}{\mathcal I}_1^*\bigl(A_\chi^{{\rm hom}, (2)}\bigr)^{-1/2}\sin\Bigl(\varepsilon^{-1}\bigl(2A_\chi^{{\rm hom},(2)}\bigr)^{1/2}t\Bigr){\mathcal I}_1&=\varepsilon\bigl(A_\chi^{\rm hom}\bigr)^{-1/2}\sin\bigl((z_\chi^-)^{1/2}t\bigr)+O\bigl(\varepsilon+\varepsilon|\chi|\bigr)\\[0.2em]
&=\varepsilon\bigl(A_\chi^{\rm hom}\bigr)^{-1/2}\sin\bigl((z_\chi^-)^{1/2}t\bigr)+O(\varepsilon).
\end{align*}

Combining this with the estimate (\ref{evolest1}) and using the formula (\ref{inverse_gelfand}) yields
\begin{equation*}
\begin{aligned}
&
\biggl\Vert(A_\varepsilon)^{-1/2}\sin\bigl((A_\varepsilon)^{1/2}t\bigr)
-\frac{1}{\sqrt{2\pi}}\int_{-\pi}^\pi\bigl(\varepsilon^{-2}A_\chi^{\rm hom}\bigr)^{-1/2}\sin\bigl((z_\chi^-)^{1/2}t\bigr)\exp({\rm i}\chi x/\varepsilon)\Theta_\chi d\chi\biggr\Vert_{L^2({\mathbb R})\to L^2({\mathbb R})}
\\[0.4em]
&\hspace{8cm}\le C\max\Bigl
\{\varepsilon^{(\alpha+2)/3}t,\varepsilon^{(4-\alpha)/6}\Bigr\}
\end{aligned}
\end{equation*}
for some $C>0.$

For $\alpha\in(1,4),$ 
the second-order approximation leads to a convergence estimate (as $\varepsilon\to0$) up to the times of order $\varepsilon^{-(\alpha+2)/3+\sigma},$ 
for all $\sigma>0.$ The corresponding error (uniform with respect to $\chi\in Y'$) is obtained from (\ref{E2}) as being of the order $O(\varepsilon^{\min\{1-(\alpha+2)/6, \sigma\}})=O(\varepsilon^{\min\{(4-\alpha)/6, \sigma\}}).$ 


\subsection{Comparison between the first-order and second-order approximations}\label{sect:comparison_bdry_triple_firstapprox_secondapprox}
\label{comparison_sec}

Within this section, we denote by $\alpha_1$ and $\alpha_2$ the values of the exponent $\alpha$ for the first-order and second-order approximations, respectively, $\alpha_1\in(1,2),$ $\alpha_2\in(1,4).$

Suppose that $\alpha_1,$ $\alpha_2$ are chosen so that the accuracies of the two approximations are the same, i.e., $(2-\alpha_1)/4=(4-\alpha_2)/6.$ Then $\alpha_2=1+3\alpha_1/2$ and the time intervals on which the approximations hold are of lengths of the orders $O(\varepsilon^{-\alpha_1+\sigma})$ and $O(\varepsilon^{-(\alpha_2+2)/3+\sigma})=O(\varepsilon^{-(1+\alpha_1/2)+\sigma}),$ for a fixed $\sigma\in(0,1).$ As $\alpha_1<2,$ it is evident that the time interval on which the second-order approximation holds is longer than that for the first-order approximation. By the same token, fixing the order of the time interval leads to a more accurate approximation in the second-order case.

\section{Concluding remarks}
    
In Sections \ref{prototype}--\ref{hyperbolic_sec},
we employed a boundary-triple framework to study long waves. This is novel in the context of homogenisation problems. We demonstrated its usage in a one-dimensional setup as a proof of concept.

We showed that if one takes an initial data $v_\text{init} \in L^2$ with an additional restriction on the support (in $\chi$) of its Gelfand transform $\hat{v}_\text{init}(y,\chi)$, then the leading-order approximation of is valid up to times $\mathcal{O}(\eps^{-2+\delta})$. 

Moreover, by keeping more terms in the Neumann series expansion in \eqref{epsM_expansion1}, plus a finer assumption on the $\chi-$support of $\hat{v}_\text{init}(y,\chi)$ (see first case of \eqref{E2}), it is possible to obtain a ``second-order approximation" (Section \ref{sect:boundary_triple_second_order}), which is an improvement on the leading-order approximation in the sense of a longer valid timescale at a common accuracy level, and in the sense of a better accuracy level at a common valid timescale (Section \ref{sect:comparison_bdry_triple_firstapprox_secondapprox}).

In connection to results of Birman-Suslina-Dorodnyi-Meshkova (Theorem \ref{thm:spectral_germ_overall}), we imposed smoothing assumptions on $v_\text{init}$, obtained a quantitative estimate in the $L^2\rightarrow L^2$ norm, and the maximal timescale in both cases are capped at the critical $\mathcal{O}(\eps^{-2})$ timescale of the classical ansatz. While this is expected based on the various ansatze discussed in Section \ref{sect:improving_basic_homo}, the present approach provides a fresh perspective in the following ways:
\begin{enumerate}[label=\roman*.]
    \item It generalizes the Birman-Suslina spectral germ to $A_\chi^{\hom}$ (Proposition \ref{Ahomprop}).
    
    \item It provides a precise link between the well-preparedness of the initial data $v_\text{init}$ and the maximal timescale.
    
    \item It expresses the second-order approximation as a \textit{single} effective \textit{self-adjoint} operator $A_\chi^{\hom,(2)}$. This is achieved by constructing a (non-unique) self-adjoint dilation of $A_\chi^{\hom}$ on $\C$ onto $\C^2$, see \eqref{gen_res}.
    
    \item By including more terms in the Neumann series expansion of the $M$-matrix $M_\chi(z)$, we have a recipe for extracting a hierarchy of operators $A_\chi^{\hom,(k)}$, potentially giving better valid effective descriptions of the hyperbolic dynamics up to the critical $\mathcal{O}(\eps^{-2})$ timescale. 
\end{enumerate}

Regarding the final point, we believe that with a more careful control of the spectral data to be kept or discarded, the boundary triple approach could be extended naturally to provide approximations beyond the $\mathcal{O}(\eps^{-2})$ timescale. This is open for future work.

\section{Acknowledgements}
KC was supported by EPSRC grant EP/V013025/1. YSL was supported by NSF grants NSF DMS-2246031 and NSF DMS-2052572. The authors are grateful to Dr Alexander V. Kiselev for insightful discussions on approximating Herglotz functions by continued fractions.

\renewcommand{\bibname}{References} 
\printbibliography

@article{kirill_igor_plates,
    author  = {K. Cherednichenko and I. Velčić},
    title   = {Sharp operator-norm asymptotics for thin elastic plates with rapidly oscillating periodic properties.},
    journal = {J. London Math. Soc.},
    year    = {2022},
    volume  = {105},
    number  = {3},
    pages   = {1634--1680}
}

@article{kirill_igor_josip_rods,
    author  = {K. Cherednichenko and I. Velčić and J. Žubrinić},
    title   = {Operator-norm resolvent estimates for thin elastic periodically heterogeneous rods in moderate contrast.},
    journal = {Calc. Var. Partial Differ. Equ.},
    year    = {2023},
    volume  = {62},
    pages   = {147}
}

@article{simplified_method,
  author        =   {Y-S. Lim and J. Žubrinić}, 
  title         =   {An operator-asymptotic approach to periodic homogenization for equations of linearized elasticity.},
  year          =   {2023},
  eprint        =   {2308.00594},
  archivePrefix =   {arXiv},
  primaryClass  =   {math.AP}
}

@article{birman_suslina_2004,
  author  = {M. Sh. Birman and T. A. Suslina}, 
  title   = {Second order periodic differential operators. Threshold properties and homogenisation.},
  journal = {St.\,Petersburg. Math. J.},
  year    = 2004,
  number  = 5,
  pages   = {639--714},
  volume  = 15
}

@article{birman_suslina_2009_hyperbolic,
  author  = {M. Sh. Birman and T. A. Suslina}, 
  title   = {Operator error estimates in the homogenization problem for nonstationary periodic equations.},
  journal = {St.\,Petersburg Math. J.},
  year    = 2009,
  number  = 6,
  pages   = {873--928},
  volume  = 20
}

@article{dorodnyi_suslina_2018_hyperbolic,
  author  = {M. A. Dorodnyi and T. A. Suslina}, 
  title   = {Spectral approach to homogenization of hyperbolic equations with periodic coefficients.},
  journal = {J. Diff. Equ.},
  year    = 2018,
  number  = 12,
  pages   = {7463--7522},
  volume  = 264
}

@article{meshkova_2021_hyperbolic,
  author  = {Y. M. Meshkova},
  title   = {On operator error estimates for homogenization of hyperbolic systems with periodic coefficients.},
  journal = {J. Spectr. Theory},
  year    = 2021,
  number  = 2,
  pages   = {587--660},
  volume  = 11
}

@article{conca_vanninathan1997,
    author = {C. Conca and M. Vanninathan},
    title = {Homogenization of periodic structures via Bloch decomposition},
    journal = {SIAM J. Appl. Math.},
    volume = {57},
    number = {6},
    pages = {1639--1659},
    year = {1997}
}

@article{conca_orive_vanninathan2002,
    author = {C. Conca and R. Orive and M. Vanninathan},
    title = {Bloch approximation in homogenization and applications.},
    journal = {SIAM J. Math. Anal.},
    volume = {33},
    number = {5},
    pages = {1166--1198},
    year = {2002}
}

@article{conca_orive_vanninathan2006,
    author = {C. Conca and R. Orive and M. Vanninathan},
    title = {On Burnett coeﬃcients in periodic media.},
    journal = {J. Math. Phys.},
    volume = {47},
    number = {3},
    pages = {11 pp.},
    year = {2006}
}

@article{santosa_symes_1991,
    author  = {F. Santosa and W. Symes},
    title   = {A dispersive effective medium for Wave propagation in periodic composites.},
    journal = {SIAM J. Appl. Math.},
    year    = {1991},
    volume  = {51},
    number  = {4},
    pages   = {984--1005}
}

@article{lamacz_1D,
    author  = {A. Lamacz},
    title   = {Dispersive effective models for waves in heterogeneous media.},
    journal = {Math. Models Methods Appl. Sci.},
    year    = {2011},
    volume  = {21},
    number  = {9},
    pages   = {1871--1899}
}

@article{abdulle_pouchon_boussinesq_trick,
    author  = {A. Abdulle and T. Pouchon},
    title   = {Eﬀective models and numerical homogenization for wave propagation in heterogeneous media on arbitrary timescales.},
    journal = {Found. Comput. Math.},
    year    = {2020},
    volume  = {20},
    number  = {},
    pages   = {1505--1547}
}

@article{allaire_briane_vanninathan_comparison,
    author  = {G. Allaire and M. Briane and M. Vanninathan},
    title   = {A comparison between two-scale asymptotic expansions and Bloch wave expansions for the homogenization of periodic structures.},
    journal = {Bol. Soc. Esp. Mat. Apl. SeMA},
    year    = {2016},
    volume  = {73},
    number  = {3},
    pages   = {237--259}
}

@article{zhikov1989,
    author = {V. V. Zhikov},
    title = {Spectral approach to asymptotic problems in diffusion},
    journal = {Diff. Equ.},
    volume = {25},
    number = {1},
    pages = {33--39},
    year = {1989}
}

@article{zhikov_pastukhova_2016_opsurvey, 
  author  = {V. V. Zhikov and S. E. Pastukhova}, 
  title   = {Operator estimates in homogenization theory.},
  journal = {Russian Math. Surv.},
  year    = 2016,
  number  = 3,
  pages   = {417-511},
  volume  = 71
}

@article{dohnal_lamacz_schweizer_2014,
    author  = {T. Dohnal and A. Lamacz and B. Schweizer},
    title   = {Bloch-wave homogenization on rarge time scales and dispersive effective wave equations.},
    journal = {Multiscale Model. Sim.},
    year    = {2014},
    volume  = {12},
    number  = {2},
    pages   = {488--513}
}

@article{allaire_lamacz_rauch_2022_crime_pays,
    author  = {G. Allaire and A. Lamacz-Keymling and J. Rauch},
    title   = {Crime pays; homogenized wave equations for long times.},
    journal = {Asymptot. Anal.},
    year    = {2022},
    volume  = {128},
    number  = {3},
    pages   = {295--336}
}

@article{benoit_gloria_2019_ballistic_transport,
  author  = {A. Benoit and A. Gloria}, 
  title   = {Long-time homogenization and asymptotic ballistic transport of classical waves.},
  journal = {Ann. Sci. \'{E}c. Norm. Sup\'{e}r.},
  year    = 2019,
  number  = 3,
  pages   = {703--759},
  volume  = 52
}

@article{duerinckx_gloria_ruf_2024_spectral_ansatz,
    author  = {M. Duerinckx and A. Gloria and M. Ruf},
    title   = {A spectral ansatz for the long-time homogenization of the wave equation.},
    journal = {J. \'{E}c. Polytech. -- Math.},
    year    = {2024},
    volume  = {11},
    number  = {3},
    pages   = {523-587}
}

@article{Physics,
    author  = {K. D. Cherednichenko and  Yu. Yu. Ershova and A. V. Kiselev},
    title   = {Time-dispersive behaviour as a feature of critical contrast media.},
    journal = {SIAM J. Appl. Math.},
    year    = {2019},
    volume  = {79},
    number  = {2},
    pages   = {690-715}
}

@article{ChEK_future,
    author  = {K. D. Cherednichenko and Yu. Yu. Ershova and A. V. Kiselev},
    title   = {Effective behaviour of critical-contrast PDEs: micro-resonances, frequency convertion, and time-dispersive properties. I.},
    journal = {Comm. Math. Phys.},
    year    = {2020},
    volume  = {375},
    pages   = {1833-1884}
}

@article{GrandePreuve,
    author  = {K. Cherednichenko and Yu. Ershova and A. Kiselev and S. Naboko},
    title   = {Unified approach to critical-contrast homogenisation with explicit links to time-dispersive media.},
    journal = {Trans. Moscow Math. Soc.},
    year    = {2019},
    volume  = {80},
    number  = {2},
    pages   = {295-342}
}

@article{CKVZ_CMP,
    author  = {K. Cherednichenko and A. V.  Kiselev and I.  Vel\v{c}i\'{c}, and J. \v{Z}ubrini\v{c}},
    title   = {Effective behaviour of critical-contrast PDEs: micro-resonances, frequency conversion, and time dispersive properties. II.},
    year          =   {2025},
    journal =   {Comm. Math. Phys.},
    volume = {406},
    pages = {72}
}

@article{Gelfand,
    author  = {I. M. Gel'fand},
    title   = {Expansion in characteristic functions of an equation with periodic coefficients. (Russian)},
    journal = {Doklady Akad. Nauk SSSR (N.S.)},
    year    = {1950},
    volume  = {73},
    pages   = {1117-1120}
}

@article{Kochubei,
    %%author  = {A. N. Ko\v cube\u\i},
    author  = {A. N. Kochubei},
    title   = {Extensions of symmetric operators and of symmetric binary relations},
    journal = {Math. Notes.},
    year    = {1975},
    volume  = {17},
    pages   = {41-48}
}

@article{Ryzhov,
    author  = {V. Ryzhov},
    title   = {Linear operators and operator functions associated with spectral boundary value problems.},
    journal = {Oper. Theory: Adv. Appl.},
    year    = {2020},
    volume  = {276},
    pages   = {576-626}
}

@article{ryzhov2020,
    author    =   {V. Ryzhov}, 
    title     =   {Linear operators and operator functions associated with spectral boundary value problems.},
    journal   = {Oper. Theory: Adv. Appl.},
    year      =   {2020},
    volume    = {276},
    pages     = {576-626}
}

@article{kirill_survey,
    author  = {K. D. Cherednichenko and Y. Y. Ershova and A. V. Kiselev and V. A. Ryzhov and L. O. Silva},
    title   = {Asymptotic analysis of operator families and applications to resonant media.},
    journal = {Oper. Theory: Adv. Appl.},
    year    = {2023},
    volume  = {291},
    pages   = {239-311}
}

@article{kirill_valery_variational,
    author  = {V. P. Smyshlyaev and K. Cherednichenko},
    title   = {On rigorous derivation of strain gradient effects in the overall behaviour of periodic heterogeneous media.},
    journal = {J. Mech. Phys. Solids},
    year    = {2000},
    volume  = {48},
    number  = {6--7},
    pages   = {1325-1357}
}

@book{zhikov,
  author    = {V. V. Zhikov and S. M. Kozlov and O. A. Oleinik}, 
  title     = {\it Homogenization of Differential Operators and Integral Functionals},
  publisher = {Springer, Berlin},
  year      = 1994,
  isbn      = {978-3-642-84661-8}
}

@book{bensoussan_lions_papanicolaou,
  author    = {A. Bensoussan and J.-L. Lions and G. Papanicolaou}, 
  title     = {\it Asymptotic Analysis for Periodic Structures},
  publisher = {North-Holland, Amsterdam},
  year      = 1978
}

@book{bakhvalov_panasenko,
  author    = {N. Bakhvalov and G. Panasenko}, 
  title     = {\it Homogenisation: Averaging Processes in Periodic Media},
  publisher = {Springer, Dordrecht},
  year      = 1989,
  series    = {Mathematics and its Applications (Soviet series)},
  isbn      = {978-94-010-7506-0}
}

@book{qshomo_book,
  author    = {S. Armstrong and T. Kuusi and J.-C. Mourrat}, 
  title     = {\it Quantitative Stochastic Homogenization and Large-Scale Regularity},
  publisher = {Springer Cham},
  year      = 2019,
  series    = {Grundlehren der mathematischen Wissenschaften},
  isbn      = {978-3-030-15544-5}
}

@book{david_borthwick,
  author    = {D. Borthwick}, 
  title     = {\it Spectral Theory},
  publisher = {Springer},
  year      = 2020,
  series    = {Graduate Texts in Mathematics},
  isbn      = {978-3-030-38001-4} 
}

@book{konrad_book,
  author    = {K. Schm\"{u}dgen}, 
  title     = {\it Unbounded Self-adjoint Operators on Hilbert Space},
  publisher = {Springer, Dordrecht},
  year      = 2012,
  series    = {Graduate Texts in Mathematics},
  isbn      = {978-94-007-9741-3} 
}

@book{reed_simon4,
  author    = {M. Reed and B. Simon}, 
  title     = {\it Methods of Modern Mathematical Physics \text{IV}: Analysis of Operators},
  publisher = {Academic Press, New York},
  year      = {1978},
}

@book{dunford_schwartz3,
  author    = {N. Dunford and J. T. Schwartz}, 
  title     = {\it Linear Operators, Part III: Spectral Operators},
  publisher = {Wiley-Interscience, New York},
  year      = {1971},
  series    = {Pure and Applied Mathematics, Vol VII}
}

@book{reed_simon1,
  author    = {M. Reed and B. Simon}, 
  title     = {\it Methods of Modern Mathematical Physics \text{I}: Functional Analysis},
  publisher = {Academic Press, New York},
  year      = {1980}
}

@book{cioranescu_donato,
  author    = {D. Cioranescu and P. Donato}, 
  title     = {\it An Introduction to Homogenization},
  publisher = {Oxford University Press},
  year      = 1999,
  series    = {Oxford Lecture Series in Mathematics and Its Applications}
}

@book{kato_book,
  author    = {T. Kato}, 
  title     = {\it Perturbation Theory for Linear Operators},
  publisher = {Springer, Berlin},
  year      = 1995,
  volume    = 132,
  edition   = 2,
  series    = {Grundlehren der mathematischen Wissenschaften},
  isbn      = {978-3-540-58661-6}
}

@book{berkolaiko_kuchment_book,
  author    = {G. Berkolaiko and P. Kuchment}, 
  title     = {\it Introduction to Quantum Graphs},
  publisher = {American Mathematical Society, Providence, RI},
  year      = 2013,
  volume    = 186,
  series    = {Mathematical Surveys and Monographs},
  isbn      = {978-0-8218-9211-4} 
}
\addcontentsline{toc}{section}{\refname} 
\end{document}